\documentclass[oneside]{m2an}

\usepackage{amsmath,amsfonts,amssymb,mathtools}
\usepackage{graphicx,subcaption}
\usepackage{braket,dsfont,url,enumerate}
\usepackage{tikz}
\usepackage{pdfpages} 
\usepackage{cmap}
\usepackage{pgfplots}
\usepackage{hyperref}
\usepackage{color,braket}
\usepackage{bbm}
\usepackage{mathtools}
\mathtoolsset{showonlyrefs}
\usepackage{fancyhdr}
\usepackage{times}
\usepackage{cmap}
\usepackage{pgf}
\usepackage{comment}
\usepackage{algorithm}
\usepackage{algorithmic}
\usepackage[textsize=small]{todonotes}
\makeatletter
\@namedef{subjclassname@1991}{2010 Mathematics Subject Classification}
\makeatother
\definecolor{darkred}{rgb}{.7,0,0}
\definecolor{red}{rgb}{0,0,0}
\definecolor{RED}{rgb}{0,0,0}

\definecolor{green}{rgb}{0,0.7,0}
\newtheorem{assumption}{Assumption}

\newcommand{\supp}{\operatorname{supp}}

\newcommand{\M}{\mathcal{M}}

\newcommand{\na}{\nabla}
\newcommand{\pa}{\partial}
\newcommand{\eps}{\varepsilon}

\newcommand{\Om}{\Omega}

\newcommand{\IOm}{I\times \Om}

\newcommand{\vertiii}[1]{{\left\vert\kern-0.25ex\left\vert\kern-0.25ex\left\vert #1
    \right\vert\kern-0.25ex\right\vert\kern-0.25ex\right\vert}}

\newcommand{\norm}[1]{\left\lVert#1\right\rVert}
\newcommand{\mnorm}[1]{\norm{#1}_{\M(\Omega)}}
\newcommand{\ltwonorm}[1]{\left\lVert#1\right\rVert_{L^2(\Omega)}}

\newcommand{\linfnorm}[1]{\lVert#1\rVert_{L^\infty(\Omega)}}

\newcommand{\abs}[1]{\lvert#1\rvert}

\newcommand{\lh}{\abs{\ln{h}}}
\newcommand{\lk}{\ln{\frac{T}{k}}}

\newcommand{\Xk}{X_k^r}

\newcommand{\Xkhs}{X^{r,s}_{k,h}}

\DeclareMathOperator*{\dist}{dist}

\newcommand{\tXk}{\widetilde X_k^r}
\newcommand{\hXk}{\widehat X_k^r}
\newcommand{\ttm}{t_{\tilde{m}}}

\newcommand{\Ltwo}{{L^2(\Omega)}}

\newcommand{\Xkr}{X_{k}^{r}}
\newcommand{\Xkhrone}{X_{k,h}^{r,1}}
\newcommand{\Xkhrtwo}{X_{k,h}^{r,2}}

\definecolor{darkgreen}{RGB}{64,191,64}

\definecolor{green}{rgb}{0,0.7,0}

\begin{document}
%
\author{Dmitriy Leykekhman}\address{Department of Mathematics,
               University of Connecticut,
              Storrs,
              CT~06269, USA (dmitriy.leykekhman@uconn.edu).}
\author{Boris Vexler}\address{Chair of Optimal Control, Technical University of Munich,
  School of Computation Information and Technology,
Department of Mathematics , Boltzmannstra{\ss}e 3, 85748 Garching b. Munich, Germany
(vexler@tum.de).}
\author{Jakob Wagner}\address{Chair of Optimal Control, Technical University of Munich,
  School of Computation Information and Technology,
Department of Mathematics , Boltzmannstra{\ss}e 3, 85748 Garching b. Munich, Germany
(wagnerja@cit.tum.de) ORCID: \url{https://orcid.org/0000-0001-8510-9790}.}

\begin{abstract}
  In this paper we analyze a homogeneous parabolic problem with initial 
  data in the space of regular Borel measures.
  The problem is discretized in time with a discontinuous Galerkin scheme of arbitrary degree
  and in space with 
  continuous finite elements of orders one or two.
  We show parabolic smoothing results for the continuous, semidiscrete and fully discrete problems.
  Our main results are interior $L^\infty$ error estimates for the evaluation at the endtime, in cases
  where the initial data is supported in a subdomain.
  In order to obtain these, we additionally show interior $L^\infty$ error estimates for $L^2$ initial data
  and quadratic finite elements, which extends the corresponding result previously established by the 
  authors for linear finite elements.
\end{abstract}
    \title{Fully discrete \textcolor{red}{pointwise} smoothing error estimates for measure valued initial data}
\maketitle

\section{Introduction}
In this work we discuss smoothing properties of the fully discrete approximation of the homogeneous parabolic problem 
\begin{equation}\label{eq:eq_aux}
\begin{aligned}
    \pa_t v-\Delta v &= 0 &&\text{in} \quad \textcolor{red}{I}\times \Om,\;  \\
    v                &= 0 &&\text{on}\quad  \textcolor{red}{I}\times\pa\Omega, \\
v(0)             &= v_0 &&\text{in}\quad \Omega,
\end{aligned}
\end{equation}
where $\Omega \subset \mathbb{R}^N$, $N=2,3$ is a bounded, convex, polygonal/polyhedral domain,
\textcolor{red}{and $I=(0,T]$ a bounded time interval.}
In particular we are interested in pointwise error estimates  in the case when the initial condition $v_0$ is a regular Borel measure $v_0\in \mathcal{M}(\Om)$ supported in some subdomain $\Om_0$ such that
$\overline \Omega_0 \subset \Omega$,
for example a linear combination of Dirac delta functions, 
$v_0=\sum_j \beta_j\delta_{x_j}$. Our main result of this paper establishes the fully discrete error estimate of the form
\begin{equation}\label{eq: smoothing global measure}
\|(v-v_{kh})(T)\|_{L^\infty(\Om_0)}\le C(\Omega_0,T)\left({k^{2r+1}}+\ell_{kh}{h^{s+1}}\right)\|{v_0}\|_{\mathcal{M}(\Om)}.
\end{equation}
Here $r\geq 0$ is the order of time discretization, $s=1,2$ is the order of the space discretization, and $\ell_{kh}$ is a logarithmic term that depends on the mesh size $h$ and the maximum time step $k$.
In order to simplify the presentation, we assume that $v_0$ is supported in the same subdomain $\Omega_0$,
in which the $L^\infty$ error is estimated, wheareas in general, two different subdomains could be chosen.
We would like to point out that the piecewise linear case $s=1$ does not require any additional smoothness assumptions beyond regularity results available on convex domains $\Om$. The higher order convergence of $s=2$ requires some additional smoothness assumptions, which are available for example on rectangular domains (cf. Section \ref{sec: higher order}). In this case, the logarithmic term $\ell_{kh}$ only depends
on $k$.

The above problem is classical and many important results are available in the literature. The $L^2$ theory for a uniform time partition is well presented in the classical textbook of Thome\'e \cite{ThomeeV_2006}. Extensions to variable time steps are available in Eriksson et al. \cite{ErikssonK_JohnsonC_LarssonS_1998a}. The $L^1\to L^\infty$ stability results, are technically more difficult and one of the first papers in this direction was the work of Schatz et al. \cite{AHSchatz_VThomee_LBWahlbin_1980a}, where such results were established in two space dimensions for piecewise linear elements and strongly A-stable single step methods with uniform time steps.  The sharpest result in the case of smooth domains and uniform time steps was obtained by A. Hansbo in \cite{HansboA_2002a}.

In our previous paper \cite{LeyVexWal19}, for piecewise linear space discretizations on a convex polygonal/polyhedral domain $\Omega$ and $v_0\in L^2(\Omega)$, we have obtained
\begin{equation}\label{eq: smoothing global}
\|(v-v_{kh})(T)\|_{L^\infty(\Om_0)}\le C(T)\left({k^{2r+1}}+\ell_{kh}{h^2}\right)\ltwonorm{v_0},
\end{equation}
with explicit form of the constant $C(T)$. Such results were required for obtaining sharp results in  initial data estimation of the parabolic problems from final time observation \cite{LeyVexWal19}. 
However, in order to extend the results  to the situation when the final time observation
is taken at a finite number of points \cite{LeyVexWagl_2023}, which is more relevant in applications, we require the results of the form \eqref{eq: smoothing global measure}.
This yields an error estimate for the adjoint state, which satisfies a backwards-in-time problem,
with a final time condition given by a measure supported in the observation points.
Since these points are fixed, this support is contained in a proper subdomain, and hence
the assumptions for \eqref{eq: smoothing global measure} are satisfied.
In summary, the main contribution of our paper is the establishment of fully discrete error estimates \eqref{eq: smoothing global measure} for Galerkin methods on potentially highly varying time partitions and quasi-uniform meshes on convex polygonal/polyhedral domains, without any additional smoothness assumptions in the case of piecewise linear case and with additional smoothness assumptions in the case of quadratic elements.    

The rest of the paper is organized as follows. In Section \ref{sec: very weak solution}, we review the notion of very weak solutions for parabolic homogeneous problems with initial data given in the space of regular
Borel measures.
In Section \ref{sec:discretization}, we discuss space-time discretization schemes and introduce the semidiscrete and fully discrete Galerkin solutions to  \eqref{eq:eq_aux}. In Section \ref{sec: parabolic smoothing}, we review and show continuous and discrete smoothing estimates for the continuous, semidiscrete and fully discrete solutions. Our main result will be the  pointwise fully discrete error estimate for initial data in $\mathcal{M}(\Om)$, see Theorem \ref{thm:fully discrete from Linfty_to_measure}, which we establish in Section \ref{sec: smoothing error}. Finally,  in Section \ref{sec: higher order} we extend our main result to a higher order space discretization. 

\section{Very weak solutions and regularity}\label{sec: very weak solution}
We begin by introducing the proper setup for the existence and regularity of the solution with measure valued initial data.
\textcolor{red}{Throughout this work, we use standard notations $L^p(\Om)$, $W^{k,p}(\Om)$, $W^{k,p}_0(\Om)$
    for the Lebesque and Sobolev spaces and abbreviate them by $H^k(\Om)$, $H^k_0(\Om)$, in case $p=2$.
    The $L^2(\Om)$ inner product will be denoted by $(\cdot,\cdot)_\Om$.
    We denote the Bochner spaces of $W^{k,p}(\Om)$ valued, $q-$integrable functions over the time interval
    $I$ by $L^q(I;W^{k,p}(\Om))$, and denote by $(\cdot,\cdot)_{I\times \Om}$
    the inner product on $L^2(I;L^2(\Om))\cong L^2(I\times \Om)$. 
    The space $\M(\Om)$ of regular Borel measures can be identified with the dual space
    of $C_0(\Om):=\{v \in C(\bar \Om): v|_{\partial \Om} = 0\}$, i.e. it holds $\M(\Om) \cong (C_0(\Om))^*$.
    The norm on $\M(\Om)$ is then given as 
    $\mnorm{\mu} := \sup_{0 \neq v \in C_0(\Om)} \frac{\langle v,\mu\rangle}{\|v\|_{C_0(\Om)}}$.
    Note that this norm is equivalent to the total variation norm $|\mu|(\Om) = \mu^+(\Om) + \mu^-(\Om)$, where
    $\mu = \mu^+-\mu^-$ is the Jordan decomposition of $\mu \in \M(\Om)$.
    With theses notations fixed, we can state the very weak formulation of \eqref{eq:eq_aux}.
}
\begin{defin}
Let $v_0 \in \M(\Om)$ be given. A function $v \in L^1(\textcolor{red}{I}\times \Om)$ is called a very weak solution to the
  heat equation \eqref{eq:eq_aux}, if 
  \begin{equation}\label{eq:veryweak_heat}
      - \int_{\textcolor{red}{I}\times \Om}
      \left( \frac{\partial \varphi}{\partial t} + \Delta \varphi \right)v \ dx \ dt
    = \int_\Om \varphi(0) \ dv_0 \qquad \text{for all } \varphi \in \Phi_T,
  \end{equation}
  where the space $\Phi_T$ of all test functions is given by
  \begin{equation}
      \Phi_T = \left\lbrace \varphi \in \Phi: \ \varphi(T) = 0 \ \text{in } \Omega \right\rbrace \quad \text{with} \quad \Phi = \left\lbrace \varphi \in L^2(\textcolor{red}{I};H^1_0(\Om)) : \  
        \textcolor{red}{\partial_t} \varphi + \Delta \varphi \in L^\infty(\textcolor{red}{I}\times \Omega)
          \land 
      \varphi(T) \in L^2(\Omega) \right\rbrace.
  \end{equation}
\end{defin}
With this definition, we have the following result (see \cite[Lemma 2.2]{CasasE_VexlerB_ZuazuaE_2015a}):
\begin{theorem}\label{thrm:existence_very_weak_solution}
  For a given $v_0 \in \M(\Om)$, there exists a unique solution $v$ in the sense of \eqref{eq:veryweak_heat}.
  The solution $v$ lies in the space $L^r(\textcolor{red}{I};W^{1,p}_0(\Om))$ for any $p,r \in [1,2)$ with 
  $\frac{2}{r} + \frac{N}{p} > N+1$, with the estimate 
  \begin{equation*}
      \|v\|_{L^r(\textcolor{red}{I};W^{1,p}_0(\Om))} \le C_{r,p} \mnorm{v_0}.
  \end{equation*}
  Moreover, $v \in C([0,T];W^{-1,p}(\Om))$, making the evaluation $v(t)$ well defined at any $t \in [0,T]$,
  and in addition $v(T) \in L^2(\Om)$ with
  \begin{equation}
    \ltwonorm{v(T)} \le C_T \mnorm{v_0}.
  \end{equation}
  For any $\varphi \in \Phi$ there holds
  \begin{equation}
    \int_\Om \varphi(T) v(T) 
    - \int_{(0,T)\times \Om} \left( \textcolor{red}{\partial_t \varphi} + \Delta \varphi \right)v \ dx \ dt
    = \int_\Om \varphi(0) \ dv_0.
  \end{equation}
\end{theorem}
In the second estimate of the above theorem, the constant $C_T$ depends on $T$. We shall make this dependence
more explicit in Lemma \ref{lem: L2 smoothign}. 
For the error analysis below, we will require the following result.
\begin{lemma}\label{lmm:composed_solution}
  Let $\tau \in (0,T)$ and let $v_1$ be the very weak solution of the heat equation on the subinterval
  $(0,\tau)$ in the sense of \eqref{eq:veryweak_heat}.
  Let $v_2$ be the weak solution of the heat equation on the subinterval $(\tau,T)$ with initial
  data $v_1(\tau) \in L^2(\Omega)$. Then $v$ defined as 
  \begin{equation}
    v(t) = \left\lbrace
      \begin{array}{ll}
      v_1(t) & t \in (0,\tau]\\
      v_2(t) & t \in (\tau,T)
      \end{array}
      \right.
  \end{equation}
  is the very weak solution in the sense of \eqref{eq:veryweak_heat}.
\end{lemma}
\begin{proof} The proof is straightforward.
\end{proof}

\section{Discretization}\label{sec:discretization}

In this section we describe the semidiscrete and  fully discrete finite element discretizations of the 
homogeneous equation~\eqref{eq:eq_aux} and present smoothing type error estimates. To discretize the problem,
we use  continuous  Lagrange finite elements of order $s\geq 1$  in space and discontinuous Galerkin methods of order $r\geq 0$ in time. 

\subsection{Time discretization}
To be more precise, we partition  $I =(0,T]$ into subintervals $I_m = (t_{m-1}, t_m]$ of length 
$k_m = t_m-t_{m-1}$, where $0 = t_0 < t_1 <\cdots < t_{M-1} < t_M =T$. The maximal and minimal time steps 
are denoted by $k =\max_{m} k_m$ and $k_{\min}=\min_{m} k_m$, respectively.
We impose the following conditions on the time mesh (as in~\cite{LeykekhmanD_VexlerB_2016a} 
or~\cite{DMeidner_RRannacher_BVexler_2011a}):
\begin{enumerate}[(i)]
  \item There are constants $c,\beta>0$ independent on $k$ such that
    \[
      k_{\min}\ge ck^\beta.
    \]
  \item There is a constant $\kappa>0$ independent on $k$ such that for all $m=1,2,\dots,M-1$
    \[
    \kappa^{-1}\le\frac{k_m}{k_{m+1}}\le \kappa.
    \]
  \item It holds $k\le\frac{T}{2r+2}$.
\end{enumerate}
The semidiscrete space $\Xk$ for the case $v_0 \in L^2(\Omega)$ is taken as 
\[
    \Xk=\Set{\varphi_k\in L^2(I;H^1_0(\Om)) | \varphi_k|_{I_m}\in 
    \mathbb{P}_{r}(I_m;H^1_0(\Om)), \ m=1,2,\dots,M},
\]
where $\mathbb{P}_{r}(I_m;V)$ is the space of polynomial functions of degree $r$ in time on $I_m$ with values
in a Banach space $V$.
However, for the semidiscrete formulation of \eqref{eq:eq_aux} with measure valued initial data $v_0\in \mathcal{M}(\Om)$ to be well defined, at initial time the test functions $\varphi_k$ need to be in $C_0(\Om)$.
Since for $N\geq 2$ the space $H^1_0(\Om)$ is not embedded in the space of continuous functions, we  need to modify the spaces of 
trial and test functions, as
\begin{equation*}
  \tXk = \left\lbrace \varphi_k \in L^2(I \times \Omega):  
 \varphi_k|_{I_m} \in \mathbb{P}_r(I_m;W_0^{1,p}(\Omega)), \ m=1,2,\dots,M
\right\rbrace
\end{equation*}
and
\begin{equation*}
  \hXk = \left\lbrace \varphi_k \in L^2(I \times \Omega) :
  \varphi_k|_{I_m} \in \mathbb{P}_r(I_m;W_0^{1,p'}(\Omega)),\  m=1,2,\dots,M
\right\rbrace,
\end{equation*}
for  some $\textcolor{red}{ \frac{2N}{N+2}} < p < \frac{N}{N-1}$ and 
$ \textcolor{red}{ \frac{2N}{N-2}} > p' > N$ the dual exponent satisfying 
$ \frac{1}{p} + \frac{1}{p'} = 1$. In this setting, the embedding 
$W_0^{1,p'}(\Om) \hookrightarrow C_0(\Om)$, yields that $\langle v_0, \varphi_{k,0}^+\rangle$ is well
defined for all \textcolor{red}{test functions} $\varphi_k \in \hXk$, 
\textcolor{red}{while every trial function $v_k \in \tXk$ satisfies
$v_k(t) \in W^{1,p}_0(\Om) \hookrightarrow L^2(\Om)$ for every $t \in I$.}
With these spaces, the dG($r$)
semidiscrete (in time)  solution $v_k$ of \eqref{eq:eq_aux} for
$v_0 \in \M(\Omega)$ is given by $v_k \in \tXk$ that satisfies
\begin{equation}\label{eq:semidiscrete_heat_measuredata}
    B(v_k,\varphi_k) = \left\langle v_0, \varphi_{k,0}^+ \right\rangle
    \quad \text{for all }\; \varphi_k \in \hXk.
\end{equation}
Here the bilinear form $B(\cdot,\cdot)$ is defined by
\begin{equation}\label{eq: bilinear form B}
  B(w,\varphi)=\sum_{m=1}^M \langle \textcolor{red}{\partial_t w},\varphi \rangle_{I_m \times \Omega} +
 (\na w,\na \varphi)_{\IOm}+\sum_{m=2}^M([w]_{m-1},\varphi_{m-1}^+)_{\Omega}+(w_{0}^+,\varphi_{0}^+)_{\Omega},
\end{equation}
where 
$\langle \cdot,\cdot \rangle_{I_m \times \Omega}$ is the duality product between $ L^2(I_m;W^{-1,p}(\Omega))$
and $ L^2(I_m;W^{1,p'}_0(\Omega))$. 
In the above definition we use the usual notation for functions with possible discontinuities at the nodes $t_m$:
\begin{equation}\label{def: time jumps}
w^+_m=\lim_{\eps\to 0^+}w(t_m+\eps), \quad w^-_m=\lim_{\eps\to 0^+}w(t_m-\eps), \quad [w]_m=w^+_m-w^-_m.
\end{equation}
\begin{remark}\label{rem:semidiscrete_l2_initial_data}
Note that whenever $v_0 \in L^2(\Omega)$ the formulation \eqref{eq:semidiscrete_heat_measuredata} is 
equivalent to searching $v_k \in \Xk$, satisfying
\begin{equation}\label{eq:semidiscrete_heat_l2}
  B(v_k,\varphi_k) = \textcolor{red}{\left( v_0, \varphi_{k,0}^+ \right)_{\Om}}
    \quad \text{for all }\; \varphi_k \in \Xk.
\end{equation}
\end{remark}
\begin{remark}
Since we are dealing with a homogeneous parabolic problem with constant coefficients,
the discontinuous Galerkin method actually coincides with subdiagonal Pad\'{e} approximations
and one can use, for example, a rational representation of the semidiscrete solution.  
\textcolor{red}{While} it is more convenient for our analysis to use the definition based on the 
bilinear form $B(\cdot,\cdot)$,
\textcolor{red}{this rational expression allows us to show wellposedness of the semidiscrete problem.}
\end{remark}
\textcolor{red}{
  \begin{theorem}
    Let $v_0 \in \M(\Om)$. Then the semidiscrete problem \eqref{eq:semidiscrete_heat_measuredata}
    has a unique solution $v_k \in \tXk$.
  \end{theorem}
  \begin{proof}
      It is sufficient to show the claim on the first time interval. 
      Let $\{\phi_j : j=0,...,r\}$ denote a basis of $\mathbb P_r(I_1;\mathbb R)$. 
      It is well known, that in the setting of \eqref{eq:semidiscrete_heat_l2},
      there exist polynomials $P_j(z)$, $j=0,...,r$ and $Q(z)$ of degrees $r$ and $r+1$, respectively,
      such that the variational formulation is equivalent to the rational representation
      \begin{equation*}
          v_k|_{I_1} = \sum_{j=0}^r \phi_j \dfrac{P_j(-k_1 \Delta)}{Q(-k_1 \Delta)} v_0,
      \end{equation*}
      see \cite{lesaint_finite_1974}, \cite[Section 4.1]{LeykekhmanD_VexlerB_2017a}.
      Here $Q(z)$ corresponds to the denominator of the subdiagonal $(r,r+1)$ Padé approximation  of $e^{-z}$.
      The polynomial $Q(z)$ posesses $r+1$ complex zeroes $\xi_n \in \mathbb C$, $n=1,...,r+1$.
      These all satisfy $Re(\xi_n) < 0$, and thus $\xi_n \in \rho(-k_1 \Delta)$, i.e., they are
      contained in the resolvent set of $-k_1 \Delta$, see \cite[Theorem 1.1]{saff_zeros_1975}.
      This implies that the operators $(\xi_n + k_1 \Delta)^{-1}\colon L^2(\Om) \to H^1_0(\Om)$,
      are well defined.
      By \cite[Theorem 8]{wanner_order_1978} and the fact that the $(r,r+1)$ Padé approximation
      is of order $2r+1$, we know that the zeroes $\xi_n$ of  $Q(z)$ are pairwise distinct.
      Hence for the coefficients of $v_k|_{I_1}$, there holds a partial fraction decomposition,
      and for some $c_{j,n} \in \mathbb C$, $j=0,...,r$, $n=1,...,r+1$ we have the representation 
      \begin{equation}\label{eq:semidiscrete_sol_representation}
          v_k|_{I_1} = \sum_{j=0}^r \phi_j \sum_{n=1}^{r+1} c_{j,n} (\xi_n + k_1 \Delta)^{-1} v_0.
      \end{equation}
      By the elliptic theory, we can show, that $(\xi_n +k_1 \Delta)^{-1}\colon \M(\Om) \to W^{1,p}_0(\Om)$,
      are well defined, which implies that \eqref{eq:semidiscrete_sol_representation} holds also true
      for $v_0 \in \M(\Om)$. 
      To show wellposedness of the elliptic problems, we employ the following construction.
      For any $\mu \in \M(\Om)$, due to 
      \cite[Corollary 2.7]{drelichman_weighted_2020}, there exists a unique solution 
      $u_\mu \in W^{1,p}_0(\Omega)$ with $\frac{2N}{N+2} < p < \frac{N}{N-1}$ to
      \begin{equation}\label{eq:elliptic_measure}
        k_1 (\nabla u_\mu,\nabla v)_\Om = \langle \mu, v\rangle \quad \text{for all } v \in W^{1,p'}_0(\Omega).
      \end{equation}
      By the embedding $W^{1,p}_0(\Om) \hookrightarrow L^2(\Om)$, it holds 
      $(\xi_n + k_1 \Delta)^{-1} u_\mu \in H^1_0(\Om)$.
      From this, we can construct $(\xi_n + k_1 \Delta)^{-1}\mu$ via 
      $\xi_n (\xi_n +k_1 \Delta)^{-1} u_\mu - u_\mu \in W^{1,p}_0(\Om)$, which concludes the proof.
\end{proof}
}

\begin{remark}\label{rem:semidiscrete_spaces}
\textcolor{red}{Due to Remark \ref{rem:semidiscrete_l2_initial_data}, the semidiscrete problem  
for $v_0 \in \M(\Om)$, can equivalently be formulated using $H^1_0(\Om)$ for test and trial functions
on the intervals $I_m$, $m=2,...M$, instead of using the spaces $\tXk,\hXk$ defined above.
By definition, it holds $v_{k,1}^- \in W^{1,p}_0(\Om) \hookrightarrow L^2(\Om)$, hence 
on subsequent intervals, the solution lies in $X_k^r$. This construction of spaces was pursued in 
\cite{LeyVexWal19}.
}
\end{remark}
Rearranging the terms in \eqref{eq: bilinear form B}, we obtain an equivalent (dual) expression for $B$:
\begin{equation}\label{eq:B_dual}
  B(w,\varphi)= - \sum_{m=1}^M \langle w,\textcolor{red}{\partial_t \varphi} \rangle_{I_m \times \Omega} +
    (\na w,\na \varphi)_{\IOm}-\sum_{m=1}^{M-1} (w_m^-,[\varphi]_m)_{\Omega} + (w_M^-,\varphi_M^-)_{\Omega}.
\end{equation}
Notice that for the very weak solution $v$ to \eqref{eq:veryweak_heat} and the semidiscrete solution $v_k \in \tXk$ to \eqref{eq:semidiscrete_heat_measuredata} we have the following orthogonality property:
\begin{equation}\label{eq: orthogonality semidiscrete}
    B(v-v_k,\varphi_k)=0  \quad \text{for all }\; \varphi_k\in \hXk,
\end{equation}
\textcolor{red}{which can be shown by splitting $v$ at $t_1$ according to Lemma \ref{lmm:composed_solution}, 
    using the weak formulation on $(t_1,T]$ and a density argument on $I_1$ to show,
    that the very weak solution can be tested with semidiscrete functions $\varphi_k \in \hXk$.}
Next we define the fully discrete approximation scheme. 
\subsection{Space discretization}
For some $h_0 > 0$ and $h \in (0, h_0]$ let $\mathcal{T}$
denote  a quasi-uniform triangulation of $\Om$  with mesh size $h$, i.e. $\mathcal{T} = \{\tau\}$ is a
partition of $\Om$ into cells (triangles or tetrahedrons) $\tau$ of diameter $h_\tau$ and measure~$|\tau|$
such that for $h=\max_{\tau} h_\tau$,
$$
h_\tau\le h \le C |\tau|^{\frac{1}{N}}, \quad \text{for all }\; \tau\in \mathcal{T},
$$
hold. Let $V^s_h\subset H^1_0(\Om)$ be the usual space of  conforming piecewise polynomial finite elements of degree $s$. We define the following three operators
to be used in the sequel: the discrete Laplacian $\Delta_h \colon V^s_h \to V^s_h$, defined by
\begin{equation}\label{eq:discrete_laplacian}
  (-\Delta_h v_h,w_h)_\Omega =(\nabla v_h,\nabla w_h)_\Omega \quad \text{for all }\; v_h,w_h \in V^s_h,
\end{equation}
the $L^2$ projection $P_h \colon L^2(\Omega) \to V^s_h$, defined by
\begin{equation}\label{eq:L2project}
(P_h v, w_h)_\Omega = ( v, w_h)_\Omega \quad \text{for all }\; w_h \in V^s_h,
\end{equation}
and the Ritz projection $R_h \colon H^1_0(\Omega) \to V^s_h$, defined by
\[
(\nabla R_h v, \nabla w_h)_\Omega = (\nabla v, \nabla w_h)_\Omega \quad \text{for all }\; w_h \in V^s_h.
\]
To obtain the fully discrete approximation of~\eqref{eq:eq_aux} we consider the space-time finite element
space
\begin{equation} \label{def: space_time}
    \Xkhs=\Set{v_{kh} \in \Xk | v_{kh}|_{I_m}\in \mathbb{P}_{r}(I_m;V^s_h), \ m=1,2,\dots,M}.
\end{equation}
We define a fully discrete cG($s$)dG($r$) approximation $v_{kh} \in \Xkhs$ of~\eqref{eq:eq_aux} by
\begin{equation}\label{eq:fully discrete heat}
    B(v_{kh},\varphi_{kh})=\left\langle v_0,\varphi_{kh,0}^+\right\rangle \quad \text{for all }\; \varphi_{kh}\in \Xkhs.
\end{equation}
Similarly to the time semidiscretization, we have the following orthogonality relation
for the \textcolor{red}{semidiscrete solution $v_k$ to \eqref{eq:semidiscrete_heat_measuredata}}
and the fully discrete solution $v_{kh} \in \Xkhs$ to \eqref{eq:fully discrete heat}:
\begin{equation}\label{eq: orthogonality fully discrete}
    B(v_k-v_{kh},\varphi_{kh})=0  \quad \text{for all }\;  \varphi_{kh}\in \Xkhs.
\end{equation}
Existence of a unique solution $v_{kh}$ is shown, e.g., in \cite{ThomeeV_2006}.
At the end of this section, we would like to introduce the  following truncation argument, 
which we will use often in our proofs. For $w_k,\varphi_k \in \Xk$,
we let $\tilde{w}_k=\chi_{(t_{\tilde{m}},T]}w_k$ and $\tilde{\varphi}_k=\chi_{(t_{\tilde{m}},T]}\varphi_{k}$, 
where $\chi_{(t_{\tilde{m}},T]}$ is the characteristic function on the interval $(t_{\tilde{m}},T]$,
for some $1\le\tilde{m}\le M$, i.e. $\tilde{w}_k=0$  on $I_1\cup \cdots \cup I_{\tilde{m}}$ for some
$\tilde{m}$  and $\tilde{w}_k={w}_k$ on the remaining time intervals.
Then from \eqref{eq: bilinear form B}, we have the identity
\begin{equation}\label{eq: Bilinear tilde}
    B(\tilde{w}_{k},\varphi_{k})=B(w_{k},\tilde{\varphi}_{k})+(w_{k,\tilde{m}}^-,\varphi_{k,\tilde{m}}^+)_\Omega.
\end{equation}
The same identity holds of course for fully discrete functions  $w_{kh},\varphi_{kh}\in\Xkhs$.

\section{Parabolic smoothing}\label{sec: parabolic smoothing}
In this section we review and establish smoothing properties of the continuous and discrete solutions, which are essential
for the establishment of our main results. 

\subsection{Smoothing estimates for the continuous problem}
It is well known that homogeneous parabolic problems have a strong smoothing effect. In particular, for  $v_0\in L^2(\Om)$, the solution  $v$ to the  problem~\eqref{eq:eq_aux} has the following smoothing property, 
see \cite[Chapter 1, Eq. 1.14]{ashyralyev_wellposedness_1994}
\begin{equation}\label{eq:continuous_smoothing}
    \|\partial_t^l v(t)\|_{L^2(\Om)} + \|(-\Delta)^l v(t)\|_{L^2(\Om)}\le
    \frac{C}{t^l}\| v_0\|_{L^2(\Om)}\quad t>0,\quad l=0,1,\dots.
\end{equation}
\begin{remark}
In many situations it is sufficient to have smoothing type estimates in $L^2$ norms and the corresponding smoothing results, for example in $L^p$ norms, can obtained by the Gagliardo-Nirenberg inequality
\begin{equation}\label{eq:general_Gagliardo-Nirenberg}
\norm{g}_{L^p(B)}\le C \norm{g}^{\alpha}_{H^2(B)} \norm{g}^{1-\alpha}_{L^2(B)},\quad 2\le p\le \infty,\quad \text{for}\ \alpha = \frac{N}{4}-\frac{N}{2p},
\end{equation}
which holds for any subdomain $B \subset \Omega$ fulfilling the cone condition (in particular for $B=\Omega$)
and for all $g \in H^2(B)$, see~\cite[Theorem 3]{AdamsFournier:1977}. In particular, for $p=\infty$ on convex domains, we have
\begin{equation}\label{eq:Gagliardo-Nirenberg_Omega}
    \norm{g}_{L^\infty(\Omega)}\le C \norm{\Delta g}^{\frac{N}{4}}_{L^2(\Omega)} 
    \norm{g}^{1-\frac{N}{4}}_{L^2(\Omega)}, \quad \text{for}\ g\in H^2(\Omega)\cap H^1_0(\Omega).
\end{equation}
 Thus using \eqref{eq:Gagliardo-Nirenberg_Omega} and the
 smoothing estimates \eqref{eq:continuous_smoothing} for $v_0\in L^2(\Om)$, we immediately obtain 
 \begin{equation}\label{eq: from_Linfty_to_L2}
  \| v(t)\|_{L^\infty(\Om)}\le
    \frac{C}{t^{\frac{N}{4}}}\| v_0\|_{L^2(\Om)}\quad t>0.
 \end{equation}
 \end{remark}
Using a duality argument and the smoothing estimates, this result can be easily extended to the case
$v_0\in \M(\Om)$.
First we derive the explicit time dependence of the constant occuring in the estimate of Theorem \ref{thrm:existence_very_weak_solution}.
\begin{lemma}\label{lem: L2 smoothign}
  Let  $v_0 \in \M(\Omega)$ and $v$ be the very weak solution of \eqref{eq:eq_aux}. Then
\begin{equation}
    \| v(t)\|_{L^2(\Om)}\le
    \frac{C}{t^{\frac{N}{4}}} \| v_0\|_{\M(\Om)}\quad t>0.
\end{equation}
\end{lemma}
\begin{proof}
We will establish the result for $t=T$. Define $y$ to be the solution to the dual problem
  \begin{equation}
    \left\lbrace 
    \begin{aligned}
      -\partial_t y - \Delta y &= 0 && \text{in } I\times \Om\\
      y &= 0 && \text{on } I \times \partial \Om\\
      y(T) &= v(T) && \text{in } \Om.
  \end{aligned}
\right.
  \end{equation}
  Then \eqref{eq: from_Linfty_to_L2} applied to $y$ yields $y(0) \in C_0(\Omega)$ and we have the estimate
  \begin{align}
    \|v(T)\|_{L^2(\Om)}^2 
    = (v(T),y(T))_\Om = \langle v_0, y(0)\rangle 
    &\le C \mnorm{v_0}\|y(0)\|_{L^\infty(\Om)}
    \le C T^{- \frac{N}{4}}\mnorm{v_0} \|v(T)\|_{L^{2}(\Om)}.
  \end{align}
 Canceling $\|v(T)\|_{L^{2}(\Om)}$ gives the result. 
\end{proof}
\begin{corr}\label{cor:continuous_smoothing_measure_data}
  Let  $v_0 \in \M(\Omega)$ and $v$ be the very weak solution of \eqref{eq:eq_aux}. Then
\begin{equation}\label{eq:continuous_smoothing_measure_data}
    \|\partial_t^l v(t)\|_{L^2(\Om)} + \|(-\Delta)^l v(t)\|_{L^2(\Om)}\le
    \frac{C}{t^{l+ \frac{N}{4}}} \| v_0\|_{\M(\Om)}\quad t>0,\quad l=0,1,\dots.
\end{equation}
\end{corr}
\begin{proof}
  This is a direct consequence of Theorem \ref{thrm:existence_very_weak_solution}, Lemma   \ref{lmm:composed_solution} and the above smoothing result. The time dependency
  of the constant can be seen, by fixing $t\in (0,T)$ and setting $\tau = \frac{t}{2}$ in Lemma 
  \ref{lmm:composed_solution}. Then by Lemma \ref{lem: L2 smoothign}, we have
  $\ltwonorm{v \left( \frac{t}{2} \right)} \le C \left(\frac{2}{t}\right)^{ \frac{N}{4}} \mnorm{v_0}$.
  By \eqref{eq:continuous_smoothing}, it also holds
    \begin{equation}
        \|\partial_t^l v(t)\|_{L^2(\Om)} + \|(-\Delta)^l v(t)\|_{L^2(\Om)}\le
        \frac{C}{(t - \frac{t}{2})^l}\ltwonorm{v \left( \frac{t}{2} \right)}
        \le \frac{C}{t^{l+ \frac{N}{4}}}\mnorm{v_0}.
    \end{equation}
\end{proof}
By applying the Gagliardo-Nirenberg  inequality \eqref{eq:Gagliardo-Nirenberg_Omega}, we immediately obtain:
\begin{corr}
  Let  $v_0 \in \M(\Omega)$ and $v$ be the very weak solution of \eqref{eq:eq_aux}. Then
\begin{equation}\label{eq:continuous_smoothing_measure_data Linfty}
    \|v(t)\|_{L^\infty(\Om)}\le
    \frac{C}{t^{ \frac{N}{2}}} \| v_0\|_{\M(\Om)},\quad t>0.
\end{equation}
\end{corr}
\subsection{Smoothing estimates for the discrete solutions}
For the time discontinuous Galerkin solutions, both the semidiscrete and the fully discrete, similar smoothing type estimates also hold (see Theorems 3,4,5,10 in
\cite{LeykekhmanD_VexlerB_2017a} and \textcolor{red}{Lemma 3.2 in \cite{LeyVexWal19}}
for general $L^p$ norms, cf. also
~\cite[Theorem~5.1]{ErikssonK_JohnsonC_LarssonS_1998a} for the case of the $L^2$ norm).
\textcolor{red}{
\begin{lemma}\label{lemm:linfty_l2_stability}
  Let $v_k \in \Xkr$ and $v_{kh} \in \Xkhs$ be the semidiscrete and fully discrete solutions  
  of \eqref{eq:semidiscrete_heat_measuredata} and \eqref{eq:fully discrete heat},
  respectively, with $v_0\in L^2(\Omega)$. Then, there exists a constant $C$ independent of $k$ and
 $h$ such that
 \begin{equation}
   \|v_k\|_{L^\infty(I;L^2(\Om))} \le C \|v_0\|_{L^2(\Om)} \quad \text{and} 
   \quad \|v_{kh}\|_{L^\infty(I;L^2(\Om))} \le C \|v_0\|_{L^2(\Om)}.
 \end{equation}
\end{lemma}
}
\begin{lemma}\label{lemma: homogeneous smoothing dG_r fully discrete}
  Let $v_k \in \Xkr$ and $v_{kh} \in \Xkhs$ be the semidiscrete and fully discrete solutions  of \eqref{eq:semidiscrete_heat_measuredata} and \eqref{eq:fully discrete heat}, respectively, with $v_0\in L^2(\Omega)$.  Then, there exists a constant $C$ independent of $k$ and
 $h$ such that
$$
\begin{aligned}
    \sup_{t\in I_m}\|\pa_t v_{k}(t)\|_{L^2(\Om)}+\sup_{t\in I_m}\|\Delta v_{k}(t)\|_{L^2(\Om)}+
        k_m^{-1}\|[v_{k}]_{m-1}\|_{L^2(\Om)}&\le \frac{C}{t_m}\|v_0\|_{L^2(\Om)},\\
    \sup_{t\in I_m}\|\pa_t v_{kh}(t)\|_{L^2(\Om)}+\sup_{t\in I_m}\|\Delta_h v_{kh}(t)\|_{L^2(\Om)}+
        k_m^{-1}\|[v_{kh}]_{m-1}\|_{L^2(\Om)}&\le \frac{C}{t_m}\|v_0\|_{L^2(\Om)},
\end{aligned}
$$
for $m=1,2,\dots,M$. For $m=1$ the jump term is understood as 
$[v_{k}]_0 = v_{k,0}^+-v_0$ and $[v_{kh}]_0 = v_{kh,0}^+-P_h v_0$.
\end{lemma}
The above estimates immediately imply the following stability result. 
\begin{corr}\label{cor: maximal parabolic initial in L1}
Under the assumptions of Lemma \ref{lemma: homogeneous smoothing dG_r fully discrete}, we have
$$
\sum_{m=1}^M \left( \|\pa_t v_{k}\|_{L^1(I_m;L^2(\Om))}+\|\Delta v_{k}\|_{L^1(I_m;L^2(\Om))}+ k_m \|\Delta v_{k,m}^+\|_{L^2(\Om)}  + \|[v_{k}]_{m-1}\|_{L^2(\Om)}\right)\le C\lk\|v_0\|_{L^2(\Om)}
$$
and
\begin{multline*}
\sum_{m=1}^M \Bigl(\|\pa_t v_{kh}\|_{L^1(I_m;L^2(\Om))}+\|\Delta_h v_{kh}\|_{L^1(I_m;L^2(\Om))}
+k_m \|\Delta_h v_{kh,m}^+\|_{L^2(\Om)}+\|[v_{kh}]_{m-1}\|_{L^p(\Om)}\Bigr)\le C\lk\|v_0\|_{L^2(\Om)}.
\end{multline*}
\end{corr}
For sufficiently many time steps, applying Lemma~\ref{lemma: homogeneous smoothing dG_r fully discrete}
iteratively, we have the following result.
\begin{lemma}\label{lemma: higher smoothing}
  Let $v_k \in \Xkr$ and $v_{kh} \in \Xkhs$ be the semidiscrete and fully discrete solutions  of \eqref{eq:semidiscrete_heat_measuredata} and \eqref{eq:fully discrete heat}, respectively, with $v_0\in L^2(\Omega)$. Then, for any $m \in \{1,2,\dots M\}$, any $l \le m$, there hold
\[
\sup_{t\in I_m}\|\pa_t(-\Delta)^{l-1} v_{k}(t)\|_{L^2(\Om)}+\sup_{t\in I_m}\|(-\Delta)^l v_{k}(t)\|_{L^2(\Om)}
+k_m^{-1}\|[(-\Delta)^{l-1}v_{k}]_{m-1}\|_{L^2(\Om)}\le \frac{C}{t^l_m}\|v_0\|_{L^2(\Om)}
\]
and
\[
    \sup_{t\in I_m}\|\pa_t(-\Delta_h)^{l-1} v_{kh}(t)\|_{L^2(\Om)}
    +\sup_{t\in I_m}\|(-\Delta_h)^l v_{kh}(t)\|_{L^2(\Om)}
    +k_m^{-1}\|[(-\Delta_h)^{l-1}v_{kh}]_{m-1}\|_{L^2(\Om)}\le \frac{C}{t^l_m}\|v_0\|_{L^2(\Om)},
\]
provided $k \le \frac{t_m}{l+1}$.
\end{lemma}
Using the continuous \eqref{eq:Gagliardo-Nirenberg_Omega} and the discrete version of the Gagliardo-Nirenberg
inequality, namely
\begin{equation}\label{eq: discrete Gagliardo-Nirenber}
    \|\chi\|_{L^\infty(\Om)}\le C\|\Delta_h \chi\|^{\frac{N}{4}}_{L^2(\Om)}\|\chi\|^{1-\frac{N}{4}}_{L^2(\Om)},
    \quad \text{for all }\; \chi\in V_h^s,
\end{equation}
which  for example was established for smooth domains in \cite[Lemma 3.3]{HansboA_2002a}, but the proof is
valid for convex domains as well, we immediately obtain the following smoothing result.
\begin{corr}\label{cor: discrete smoothing in Linfty}
Under the assumptions of Lemmas \textcolor{red}{\ref{lemm:linfty_l2_stability} and}
\ref{lemma: homogeneous smoothing dG_r fully discrete} for all
$m=\textcolor{red}{1,2,...,M}$, we have
$$
   \sup_{t\in I_m}\| v_{k}(t)\|_{L^\infty(\Om)}\le \frac{C}{t_m^{N/4}}\|v_0\|_{L^2(\Om)}\quad\text{and} \quad \sup_{t\in I_m}\| v_{kh}(t)\|_{L^\infty(\Om)}\le \frac{C}{t_m^{N/4}}\|v_0\|_{L^2(\Om)}.
$$
\end{corr}
Similarly to the continuous case, 
using a duality argument, the above smoothing results can be extended to $v_0\in \M(\Om)$.

\begin{lemma}\label{lemma: smoothing 2}
  Let \textcolor{red}{$v_0 \in \M(\Om)$, and let $v_k \in \tXk$} and $v_{kh} \in \Xkhs$
  be the semidiscrete and the 
  fully discrete solutions of \eqref{eq:semidiscrete_heat_measuredata} and \eqref{eq:fully discrete heat} 
  respectively. For any $m \in \{1,2,\dots M\}$,  there hold
 \begin{equation}
       \|v_k(t_m)\|_{L^2(\Om)}+  \|v_{kh}(t_m)\|_{L^2(\Om)}\le
    \frac{C}{t^{\frac{N}{4}}_m}\|v_0\|_{\mathcal{M}(\Om)}.
 \end{equation}
\end{lemma}
\begin{proof}
  \textcolor{red}{Let $m=1,2,\dots,M$, and}
define $y_k\in \hXk$ to be the semidiscrete solution of the backward problem
$$
B(\psi_k,y_k)=(\psi_{k,m}^-,v_k(t_m))_\Omega, \quad \forall \psi_k\in \tXk,
$$
\textcolor{red}{where the right hand side is well defined, due to the assumptions on 
  $\tXk$, yielding $v_k(t_m) \in L^2(\Om)$.
  Since for this dual problem, the test functions are taken from the weaker space $\tXk$,
choosing $\psi_k=v_k \in \tXk$,}
and using Corollary \ref{cor: discrete smoothing in Linfty} for the backward problem, we have
$$
\|v_k(t_m)\|^2_{L^2(\Om)} = B(v_k,y_k)=\left\langle v_0, y_{k,0}^+ \right\rangle\le \|v_0\|_{\mathcal{M}(\Om)}\|y_{k}(0)\|_{L^\infty(\Om)}\le \frac{C}{t_m^{N/4}}\|v_0\|_{\mathcal{M}(\Om)}\|v_k(t_m)\|_{L^2(\Om)}.
$$
Canceling, we obtain the result for the time semidiscrete solution $v_k$.
The argument for the fully discrete solution $v_{kh}$ is almost identical. 
\end{proof}

From Lemma \ref{lemma: higher smoothing}, we can obtain the following result
\begin{lemma}\label{lemma: higher smoothing 2}
  Let $v_k \in \textcolor{red}{\tXk}$ and $v_{kh} \in \Xkhs$ be the semidiscrete and the fully discrete solutions of \eqref{eq:semidiscrete_heat_measuredata} and 
  \eqref{eq:fully discrete heat} respectively.
 Let $m \in \{1,2,\dots M\}$ large enough and $l\le m$,
 such that $k \le \min \lbrace \frac{t_m}{4}, \frac{t_m}{2(l+1)}\rbrace$, then there hold
 \begin{equation}
    \sup_{t\in I_m}\|\pa_t(-\Delta)^{l-1} v_{k}(t)\|_{L^2(\Om)}+
    \sup_{t\in I_m}\|(-\Delta)^l v_{k}(t)\|_{L^2(\Om)}+
    \sup_{t\in I_m}k_m^{-1}\|[(-\Delta)^{l-1}v_{k}]_{m-1}\|_{L^2(\Om)}\le
    \frac{C}{t^{l+\frac{N}{4}}_m}\|v_0\|_{\mathcal{M}(\Om)}
 \end{equation}
 and
 \begin{equation}
    \sup_{t\in I_m}\|\pa_t(-\Delta_h)^{l-1} v_{kh}(t)\|_{L^2(\Om)}+
    \sup_{t\in I_m}\|(-\Delta_h)^l v_{kh}(t)\|_{L^2(\Om)}+
    \sup_{t\in I_m}k_m^{-1}\|[(-\Delta_h)^{l-1}v_{kh}]_{m-1}\|_{L^2(\Om)}\le
    \frac{C}{t^{l+{\frac{N}{4}}}_m}\|v_0\|_{\mathcal{M}(\Om)}.
 \end{equation}
\end{lemma}
\begin{proof}
We will only establish semidiscrete smoothing estimates for measure valued initial data, the analysis for the fully discrete solution is similar. 
Combining Lemma \ref{lemma: higher smoothing} with Lemma \ref{lemma: smoothing 2},
 gives us for all $m > \tilde{m}+l$:
$$
\begin{aligned}
&\sup_{t\in I_m}\|\pa_t(-\Delta)^{l-1} v_{k}(t)\|_{L^2(\Om)}
    +\sup_{t\in I_m}\|(-\Delta)^l v_{k}(t)\|_{L^2(\Om)}
    +k_m^{-1}\|[(-\Delta)^{l-1}v_{k}]_{m-1}\|_{L^2(\Om)}\\ 
    \le & \frac{C}{(t_m - \ttm)^l}\|v_k(\ttm)\|_{L^2\Om)}\le  \frac{C}{(t_m - \ttm)^{l} \ \ttm^{\frac{N}{4}}} \mnorm{v_0}.
\end{aligned}
$$
For fixed $t_m$ with $m$ large enough such that 
$k \le \min \lbrace \frac{t_m}{4}, \frac{t_m}{2(l+1)}\rbrace$ we apply the above argument to $\ttm$ 
such that $ \frac{t_m}{2} \in (t_{\tilde{m}-1},\ttm]$. By the requirements on $k$ we obtain on the one hand
that there are at least $l$ timesteps between $\ttm$ and $t_m$, which allows us to use the smoothing
estimate of Lemma \ref{lemma: higher smoothing} for the specified $l$. On the other hand, we obtain
\[
  t_m - \ttm \ge \frac{t_m}{2}-k \ge \frac{t_m}{2} - \frac{t_m}{4} = \frac{t_m}{4}
  \qquad \text{yielding} \qquad
  (t_m - \ttm)^{-l} \le 4^l \frac{1}{t_m^l}.
\]
The choice $ \frac{t_m}{2} \le \ttm$ gives $\ttm^{- \frac{N}{4}} \le 
2^{\frac{N}{4}}t_m^{- \frac{N}{4}}$, which allows us to eliminate
$\ttm$ in the final bound and obtain
\begin{equation}
    \sup_{t\in I_m}\|\pa_t(-\Delta)^{l-1} v_{k}(t)\|_{L^2(\Om)}
    +\sup_{t\in I_m}\|(-\Delta)^l v_{k}(t)\|_{L^2(\Om)}
    +k_m^{-1}\|[(-\Delta)^{l-1}v_{k}]_{m-1}\|_{L^2(\Om)}
    \le C(l,N)t_m^{-l-\frac{N}{4}} \mnorm{v_0}.
\end{equation}
\end{proof}

\section{Smoothing type error estimates}\label{sec: smoothing error}
First we review smoothing results with the initial data in $L^2(\Omega)$ and then extend the corresponding results to ${\mathcal{M}(\Om)}$.

\subsection{Review of pointwise smoothing error estimates for $v_0\in L^2(\Om)$}
In \cite{LeyVexWal19}, we have established the following pointwise fully discrete error estimate.
\begin{prop} \label{prop:fully discrete from Linfty_to_L2}
 Let $v_0\in L^2(\Omega)$, let $v$ and $v_{kh} \in \Xkhrone$ satisfy \eqref{eq:eq_aux}
  and \eqref{eq:fully discrete heat}, respectively. Then for any subdomain $\Omega_0$ with $\overline \Omega_0\subset \Omega$ there holds
\[
    \|(v-v_{kh})(T)\|_{L^\infty(\Om_0)} \le C(T,\Om_0)\left(\ell_{kh} h^2+k^{2r+1} \right)\ltwonorm{v_0},
\]
where $\ell_{kh} = \lk + \lh$ and $C(T,\Om_0)$ is a constant that depends on $T$ and $\Om_0$ and the explicit form can be traced from the proof.
\end{prop}
The  proof of the above result was based on the following splitting  of the error
\begin{equation}\label{eq:error_splitting_l2}
(v-v_{kh})(T) = (v-v_{k})(T)+(R_hv_k-v_{kh})(T)+(v_k-R_hv_{k})(T).
\end{equation}
Then each term was treated separately. 
The first error term was estimated in \cite[Theorem 3.8]{LeyVexWal19} by
\begin{equation}\label{eq:semidiscrete_error}
    \linfnorm{(v-v_k)(T)} \le C(T) k^{2r+1}\ltwonorm{v_0},
\end{equation}
with $C(T) \sim T^{-(2r+1+\frac{N}{4})}$. The above estimate follows from
 (see \cite[Lem.~7.2]{LeyVexWal19})
\begin{equation}\label{eq: Delta_j v-vk in L2}
\ltwonorm{(-\Delta)^j(v-v_k)(T)} \le C_j(T) k^{2r+1}\ltwonorm{v_0},\quad j=0,1,\dots.
\end{equation}
The second error term in \eqref{eq:error_splitting_l2} satisfies,
\begin{equation}\label{eq: estimates R_hvk-vkh in PL}
   \|(R_hv_k-v_{kh})(T)\|_{L^\infty(\Om)} \le C(T)\lk h^2 \ltwonorm{v_0},
\end{equation}
which followed from (see \cite[Lem.~8.2--8.3]{LeyVexWal19})
\begin{equation}
\|(-\Delta_h)^j(R_hv_k-v_{kh})(T)\|_{L^2(\Om)} \le C_j(T)\lk h^2 \ltwonorm{v_0},\quad j=0,1,
\end{equation}
and the discrete  Gagliardo-Nirenberg  inequality 
\eqref{eq: discrete Gagliardo-Nirenber}. Here, we point out that the treatment of the first and the second
terms of \eqref{eq:error_splitting_l2} do not require the condition $\overline \Omega_0\subset \Omega$,
they are global in nature. 
Finally, the estimate of the last
term in \eqref{eq:error_splitting_l2} follows from the interior elliptic error estimate (cf. \cite{SchatzWahlbin_1977})
\begin{equation}
\|(v_k-R_hv_k)(T)\|_{L^\infty(\Om_0)} \le C(T,\Om_0)\lh h^2 \ltwonorm{v_0}.
\end{equation}

\subsection{Pointwise smoothing error estimates for $v_0\in \M(\Om)$}\label{subsec:pointwise_smoothing_meas}
We now turn towards proving the pointwise error estimate for measure valued initial data.
To this end, first recall that in \cite[Lemma 5.1]{LeyVexWal19} we have shown the following
$L^2$ error estimate for parabolic problems with initial data in $\M(\Om)$, where for the spatial
estimate we impose a condition on the support of $v_0$.
\begin{lemma}\label{lemm:L2error_measuredata}
  Let $v_0 \in \M(\Om)$ with $\supp v_0 \subset \Om_0$ for some subdomain 
  $\Om_0 \subset \overline \Om_0 \subset \Om$ and let $v$, $v_k \in \tXk$ and $v_{kh} \in \Xkhrone$ 
  the continuous, semidiscrete and fully discrete solutions to \eqref{eq:eq_aux},
  \eqref{eq:semidiscrete_heat_measuredata} and \eqref{eq:fully discrete heat} respectively.
  Then there hold the estimates
  \begin{align*}
    \|(v - v_k)(T)\|_{L^2(\Om)} & \le C(T) k^{2r+1}\|v_0\|_{\M(\Om)}\\
    \|(v_k - v_{kh})(T)\|_{L^2(\Om)} & \le C(\Omega_0,T) \ell_{kh} h^2\|v_0\|_{\M(\Om)},
  \end{align*}
where $\ell_{kh} = \lk + \lh$ and $C(T,\Omega_0)$ is a constant that depends on $T$ and $\Omega_0$ and the explicit form can be traced from the proof.
\end{lemma}
Our main result can now be obtained directly by introducing an auxiliary solution
and the smoothing results presented in Section \ref{sec: parabolic smoothing}.
We first prove the error estimate for the spatial discretization. The proof of the error estimate for the 
time semidiscretization follows the same steps under milder assumptions, see Lemma 
\ref{thm:smoothing_est_time} below.
\begin{theorem} \label{thm:fully discrete from Linfty_to_measure}
 Let $v_0\in \M(\Omega)$, let $v_k \in \tXk$ and $v_{kh} \in \Xkhrone$ satisfy 
 \eqref{eq:semidiscrete_heat_measuredata}
  and \eqref{eq:fully discrete heat}, respectively. Then for any subdomain $\Omega_0$ with 
  $\overline \Omega_0 \subset \Omega$ and $\supp v_0 \subset \Omega_0$ there holds
\[
    \|(v_k-v_{kh})(T)\|_{L^\infty(\Om_0)} \le C(T,\Omega_0)\ell_{kh} h^2\|v_0\|_{\mathcal{M}(\Om)},
\]
where $\ell_{kh} = \lk + \lh$ and $C(T,\Omega_0)$ is a constant that depends on $T$ and $\Omega_0$ and the explicit form can be traced from the proof.
\end{theorem}
\begin{proof}
  As done in the proofs of the smoothing results, we begin by splitting the time interval. To this end
  let $\tilde m$ be such that $ \frac{T}{2} \in I_{\tilde m}$.
  We introduce a fully discrete auxiliary state $\hat v_{kh} \in \Xkhrone$, defined by
  \begin{equation*}
    B(\hat v_{kh},\varphi_{kh}) = (v_{k,\tilde m-1}^-,\varphi_{kh,\tilde m-1}^+)_\Omega \quad 
    \text{for all} \quad \varphi_{kh} \in \Xkhrone.
  \end{equation*}
  Note that by definition $\hat v_{kh} \equiv 0$ on $I_1 \cup ... \cup I_{\tilde m -1}$ and it
  satisfies a discrete problem on $I_{\tilde m} \cup ... \cup I_M$  
  with initial condition $v_{k,\tilde m-1}^-$ at time $t_{\tilde m -1}$.
  By the triangle inequality, we obtain
  \begin{equation}
    \|(v_k-v_{kh})(T)\|_{L^\infty(\Om_0)}
    \le \|(v_k- \hat v_{kh})(T)\|_{L^\infty(\Om_0)}
    + \|(\hat v_{kh} - v_{kh})(T)\|_{L^\infty(\Om_0)},
  \end{equation}
  where for the first term, we obtain with Proposition \ref{prop:fully discrete from Linfty_to_L2}
  and the semidiscrete parabolic smoothing result of Lemma \ref{lemma: smoothing 2}
  \begin{align}
    \|(v_k- \hat v_{kh})(T)\|_{L^\infty(\Om_0)}
    & \le C(T-t_{\tilde m -1},\Omega_0) \ell_{kh} h^2 \|v_{k,\tilde m-1}^-\|_{L^2(\Omega)}\\
    & \le C(T-t_{\tilde m -1},\Omega_0) t_{\tilde m -1}^{- \frac{N}{4}} \ell_{kh} 
    h^2\|v_0\|_{\M(\Omega)}.
  \end{align}
  For the second error term, we observe that the difference $\hat v_{kh} - v_{kh}$ satisfies 
  a fully discrete parabolic equation on the intervals $I_{\tilde m} \cup ... \cup I_M$ for the initial
  data $v_{k,\tilde m -1}^- - v_{kh,\tilde m -1}^-$. Hence, the discrete Gagliardo-Nirenberg inequality
  \eqref{eq: discrete Gagliardo-Nirenber} and the fully discrete smoothing results of Lemmas 
  \ref{lemma: homogeneous smoothing dG_r fully discrete} and \ref{lemma: higher smoothing} yield
  \begin{align*}
    \|(\hat v_{kh} - v_{kh})(T)\|_{L^\infty(\Om_0)} 
    & \le C\|(\hat v_{kh} - v_{kh})(T)\|_{L^2(\Om)}^{ \frac{1}{2} }
    \|\Delta_h (\hat v_{kh} - v_{kh})(T)\|_{L^2(\Om)}^{ \frac{1}{2}}\\
    & \le C (T-t_{\tilde m -1})^{-\frac{1}{2} - \frac{N}{4}}
    \|v_{k,\tilde m -1}^- - v_{kh,\tilde m -1}^-\|_{L^2(\Omega)}
  \end{align*}
  We apply Lemma \ref{lemm:L2error_measuredata} in order to estimate the $L^2$-error of the full discretization
  at the intermediate point in time, which yields
  \begin{equation}
    \|(\hat v_{kh} - v_{kh})(T)\|_{L^\infty(\Om_0)}
    \le C(t_{\tilde m-1},\Omega_0) (T-t_{\tilde m -1})^{-\frac{1}{2} - \frac{N}{4}} 
    \ell_{kh}h^2\|v_0\|_{\M(\Omega)}.
  \end{equation}
  Since the assumptions on $k$ and $\tilde m$ yield $\frac{T}{4} \le t_{\tilde m -1} \le \frac{T}{2}$ and 
  $\frac{T}{2} \le T- t_{\tilde m -1} \le \frac{3T}{4}$, as before we can replace all quantities involving $t_{\tilde m-1}$
  by ones only dependent of $T$, which concludes the proof.
\end{proof}
 Note that by exactly the same technique, we can also derive the corresponding error estimate for the
 semidiscrete problem, which is global in $\Omega$ and no constraint on $\supp v_0$ is required.
 This is due to the fact,
 that the semidiscrete results of \eqref{eq:semidiscrete_error} and \cite[Lemma 5.1]{LeyVexWal19} 
 hold in this more general setting. There holds the following result.

\begin{lemma}\label{thm:smoothing_est_time}
Let $v_0\in \mathcal{M}(\Om) $ and $v$ and $v_k \in \tXk$ be the very weak and semidiscrete solutions to
\eqref{eq:veryweak_heat} and \eqref{eq:semidiscrete_heat_measuredata} respectively.
Then there holds
\[
    \linfnorm{(v-v_k)(T)} \le C(T) k^{2r+1}\|v_0\|_{\mathcal{M}(\Om)},
\]
with $C(T) \sim T^{-(2r+1+\frac{N}{2})}$.
\end{lemma}

\section{Higher order space discretizations}\label{sec: higher order}
Our main result from the previous section, Theorem \ref{thm:fully discrete from Linfty_to_measure}, was established for piecewise linear finite elements only and does not require any additional smoothness assumptions on the solutions beyond $H^2$ regularity that is provided by the convexity of the domain. If additional regularity is available, for example, 
\begin{equation}\label{ass:regularity}
    |v|_{H^3(\Omega)} \le C \|\Delta v\|_{H^1_0(\Omega)}
\end{equation}
for any $v \in H^1_0(\Omega)$ with $\Delta v \in H^1_0(\Om)$,
then the results of Proposition \ref{prop:fully discrete from Linfty_to_L2}
can be extended (with an improved rate)
 to the case of
quadratic Lagrange finite elements which we will denote by $V_h^2$ in this section. 
\begin{remark}\label{rem:bc_laplacian_semidiscrete}
  Since due to Remark \ref{rem:semidiscrete_spaces} for each $t \in (t_1,T]$, the solution 
  $v_k$ to the semidiscrete problem (\ref{eq:semidiscrete_heat_measuredata}) 
  satisfies $v_k(t) \in H^1_0(\Om)$, one can also show straightforwardly that 
  \begin{align*}
    \Delta v_k(t), \ \partial_t \Delta v_k(t) \in H^1_0(\Omega) \ \text{ for all } t\in I_m, \ m \ge 2 
    \quad \text{and} \quad
    \Delta^2 v_k(t) \in H^1_0(\Omega) \ \text{ for all } t\in I_m, \ m \ge 3.
  \end{align*}
\end{remark}

Additional regularity is available on special domains, for example on rectangles,
right or equilateral triangles. We make the following assumption of the domain $\Omega$.

\begin{assumption}\label{assump:regularityH3}
For every $u \in H^1_0(\Omega)$ with $\Delta u \in H^1_0(\Omega)$ there holds $u \in H^3(\Omega)$. Moreover, there exists a constant $C$ independent of $u$ such that 
\begin{equation}\label{ass:regularity 2}
    \|u\|_{H^{3}(\Omega)} \le C \|\na\Delta u\|_{L^2(\Omega)}.
\end{equation}
\end{assumption}
\begin{example}
This assumption holds for example on a rectangle, see \cite[Lemma 2.4]{HellOstermannSandbichler:2015}.
In this case the solution $u$ to the elliptic equation 
\begin{align*}
	-\Delta u &=f\quad \text{in}\ \Omega\\
	u &= 0\quad \text{on}\ \partial\Omega,
\end{align*}
with $f \in H^1_0(\Omega)$ possesses the $H^3(\Omega)$ regularity with the estimate
\[
\|u\|_{H^{3}(\Omega)} \le C \|\na f\|_{L^2(\Omega)}
\]
holds. Thus,  Assumption \ref{assump:regularityH3} is satisfied in this case. 
\end{example}

\begin{lemma}\label{thm:h3regularity}
Let $\Omega$ satisfy Assumption \ref{assump:regularityH3}.
\begin{enumerate}
\item Let $u,\Delta u\in H^1_0(\Omega)$, and $\Delta^2 u \in L^2(\Omega)$. Then there holds
\begin{equation}\label{eq:est:H3}
\|u\|^2_{H^{3}(\Omega)} \le C \|\Delta u\|_{L^2(\Omega)}\|\Delta^2 u\|_{L^2(\Omega)}.
\end{equation}
\item Let $\Omega_0$ be a subdomain with $\overline \Omega_0 \subset \Omega$, let $u,\Delta u, \Delta^2 u\in H^1_0(\Omega)$, and $\Delta^3 u \in L^2(\Omega)$. Then there holds
\begin{equation}\label{eq:est:H5}
\|u\|^2_{H^{5}(\Omega_0)} \le C \|\Delta^2 u\|_{L^2(\Omega)}\|\Delta^3 u\|_{L^2(\Omega)}.
\end{equation}
\end{enumerate}
\end{lemma}
\begin{proof}
\begin{enumerate}
\item For $v\in H^1_0(\Omega)$ with $\Delta v \in L^2(\Omega)$ one directly obtains
\[
\|\nabla v\|^2_{L^2(\Omega)} \le \|v\|_{L^2(\Omega)} \|\Delta v\|_{L^2(\Omega)}.
\]
Due to $\Delta u \in H^1_0(\Omega)$ and $\Delta^2 u \in L^2(\Omega)$ this inequality can be applied to $v = \Delta u$ leading to
\[
\|\nabla \Delta u\|^2_{L^2(\Omega)} \le \|\Delta u\|_{L^2(\Omega)} \|\Delta^2 u\|_{L^2(\Omega)}.
\]
Thus, Assumption \ref{assump:regularityH3} implies the desired estimate.
\item Using a higher interior regularity result, see \cite[Chapter 6.3, Theorem 2]{evans_partial_2010},
  we obtain
\[
  \|u\|_{H^{5}(\Omega_0)} \le C(\|\Delta u\|_{H^{3}(\textcolor{red}{\Omega})} + \|u\|_{L^2(\Omega)}).
\]
Since $\Delta^2 u \in H^1_0(\Omega)$ and $\Delta^3 u \in L^2(\Omega)$ we can apply \eqref{eq:est:H3} to $\Delta u$ leading to
\[
\|\Delta u\|^2_{H^{3}(\Omega)} \le C \|\Delta^2 u\|_{L^2(\Omega)}\|\Delta^3 u\|_{L^2(\Omega)}.
\]
This leads to
\[
\|u\|^2_{H^{5}(\Omega_0)} \le C \|\Delta^2 u\|_{L^2(\Omega)}\|\Delta^3 u\|_{L^2(\Omega)} + C\|\Delta u\|_{L^2(\Omega)}\|\Delta^2 u\|_{L^2(\Omega)},
\]
which proves the desired result by $\|\Delta^j u\|_{L^2(\Omega)} \le C\|\Delta^{j+1} u\|_{L^2(\Omega)}$ 
for $j=1,2$. 
\end{enumerate}
\end{proof}

Complementing the standard error estimates for the Ritz projection in the $L^2$ and $H^1$ norms, under Assumption \ref{assump:regularityH3},
we also have the following negative norm estimate. Note that even though no $H^3$ regularity of the 
solution $u$ is used explicitly in the 
estimates, the duality argument used to prove the result, requires the assumption to hold true for 
any $H^1_0$ right hand side.
\begin{lemma}\label{lemm:H-1estimate}
  Let $u \in H^1_0(\Omega)$ and Assumption \ref{assump:regularityH3} hold true. Then it holds
  \begin{equation*}
    \|u - R_h u \|_{H^{-1}(\Omega)} \le C h^2 \| \nabla (u - R_h u) \|_{L^2(\Omega)}.
  \end{equation*}
  If further $u \in H^2(\Omega)$, then it holds 
  \begin{equation*}
    \|u - R_h u \|_{H^{-1}(\Omega)} \le C h^3 \|u \|_{H^{2}(\Omega)}
    \le C h^3 \|\Delta u \|_{L^{2}(\Omega)}.
  \end{equation*}
\end{lemma}
\begin{proof}
  The first estimate is proved by a duality argument in \cite[Theorem 5.8.3]{BrennerScott}.
  The second estimate then follows with the standard $H^1$ error estimate and $H^2$ regularity.
\end{proof}

Under Assumption \ref{assump:regularityH3} we can establish the main results of this section.
We first consider again the case of $L^2$ initial data. The extension to $v_0 \in \M(\Omega)$
then follows analogously to the case of linear finite elements.
\begin{theorem}\label{thm:smoothing_est_space_p2}
Let $v_0\in L^2(\Omega)$, let $v$ and $v_{kh} \in \Xkhrtwo$ satisfy \eqref{eq:eq_aux} and 
\eqref{eq:fully discrete heat}, respectively. Then for any subdomain $\Omega_0$ with $\overline \Omega_0\subset \Omega$ there holds
\begin{equation}\label{eq:h3_regularity}
    \|(v-v_{kh})(T)\|_{L^\infty(\Om_0)} \le C(T,d)\left(\ell_{k}h^3 + k^{2r+1} \right)\ltwonorm{v_0},
\end{equation}
where $\ell_{k} = \lk$, $d=\dist(\Om_0,\partial \Om)$
and $C(T,d)$ is a constant depending on $T$ and $d$.
\end{theorem}
\subsection{Proof of Theorem \ref{thm:smoothing_est_space_p2}}
The exact dependence of the constant $C$ on $T$ and $d$ is available in the proof of this result.
The rest of this section is devoted to the establishment of the above theorem.
The  proof for the quadratic case is similar to the proof for the piecewise linear case, but requires some modifications. As it was done in \cite{LeyVexWal19}, we split the error as 
\begin{equation}\label{eq:err_split_quadFE}
(v-v_{kh})(T) = (v-v_{k})(T)+(v_k-R_hv_{k})(T)+(R_hv_k-v_{kh})(T):=T_1+T_2+T_3.
\end{equation}
The first time semidiscrete term $T_1$ is already estimated in \cite[Theorem 3.8]{LeyVexWal19}. The second
term $T_2$ can again be estimated by the interior pointwise error estimates of 
\cite[Theorem 5.1]{SchatzWahlbin_1977},
\begin{equation}\label{eq: from SW1977 2}
\|(v_k - R_h v_k)(T)\|_{L^\infty(\Om_0)}\le C \|v_k(T)-\chi\|_{L^\infty(\Om_d)}+Cd^{-N/2}\|(v_k - R_h v_k)(T)\|_{L^2(\Om)},
\end{equation}
for any $\chi\in V^2_h$, where $\Om_d$ is a subdomain satisfying $\overline \Om_0\subset\Om_d\subset \overline \Om_d \subset\Om$ 
and $d=\dist(\Omega_0,\partial \Omega_d)$. We note that in contrast to the linear elements, in the above estimate the logarithmic term is not needed.
By the approximation theory and the Sobolev embedding
  $H^5(\Omega_d) \hookrightarrow W^{3,\infty}(\Omega_d)$  (see e.g.
  \cite[Theorem 4.12]{Adams_Fournier_2005}), 
  Lemma \ref{thm:h3regularity} and the discrete parabolic smoothing result of Lemma
  \ref{lemma: homogeneous smoothing dG_r fully discrete}, we obtain
$$
\begin{aligned}
\|v_k(T)-\chi\|_{L^\infty(\Om_d)}&\le Ch^3 \|v_k(T)\|_{W^{3,\infty}(\Omega_d)} \le  Ch^3 \|v_k(T)\|_{H^5(\Omega_d)}\\
& \le
      Ch^3\|\Delta^2 v_k(T)\|^{ \frac{1}{2}}_{\Ltwo}\|\Delta^3 v_k(T)\|^{ \frac{1}{2}}_{\Ltwo}
      \le \frac{Ch^3}{T^{\frac{5}{2}}}\|v_0\|_{L^2(\Om)}.
\end{aligned}
$$
The pollution term $\|(v_k - R_h v_k)(T)\|_{L^2(\Om)}$ from \eqref{eq: from SW1977 2}, can be estimated using global  elliptic estimates in $L^2$ norm, Lemma \ref{thm:h3regularity} and 
Lemma \ref{lemma: higher smoothing 2} as
\begin{equation}
\|(v_k - R_h v_k)(T)\|_{L^2(\Om)}\le Ch^3\|v_k(T)\|_{H^3(\Om)}
\le Ch^3\|\Delta v_k(T)\|^{ \frac{1}{2}}_{\Ltwo}\|\Delta^2 v_k(T)\|^{ \frac{1}{2}}_{\Ltwo}
\le \ \frac{Ch^3}{T^{\frac{3}{2}}}\|v_0\|_{L^2(\Om)}.
\end{equation}
Thus,
$$
\|(v_k - R_h v_k)(T)\|_{L^\infty(\Om_0)}\le C(T,\Omega_0)h^3\|v_0\|_{L^2(\Om)},
$$
and it remains to estimate the last term $T_3$ of \eqref{eq:err_split_quadFE}.
As done in \cite[Lemmas~8.2--8.3]{LeyVexWal19}, this will be achieved by estimating
\begin{equation}\label{eq:Rhvk_vkh}
\|(-\Delta_h)^j(R_hv_k-v_{kh})(T)\|_{L^2(\Om)},\quad j=0,1,
\end{equation}
and the discrete  Gagliardo-Nirenberg  inequality 
\eqref{eq: discrete Gagliardo-Nirenber}.
The proof of the above estimates was facilitated by the following technical lemma,
see \cite[Lemma 8.1]{LeyVexWal19}.
\begin{lemma}\label{lemma: P_h u_k-u_kh}
 Let $v_0\in L^2(\Omega)$, let $v$ and $v_{kh} \in \Xkhrone$ satisfy \eqref{eq:eq_aux}
  and \eqref{eq:fully discrete heat}, respectively.
There exists a constant $C$ independent of $k$, $h$, and $T$ such that
$$
    \|\Delta_h^{-1}(P_hv_k-v_{kh})(T)\|_{L^2(\Om)}\le Ch^{2} \lk\|v_0\|_{L^2(\Om)}.
$$
\end{lemma}
In order to prove Theorem \ref{thm:smoothing_est_space_p2} we thus first extend Lemma 
\ref{lemma: P_h u_k-u_kh} to quadratic finite elements in space, in order to estimate
the terms of \eqref{eq:Rhvk_vkh}.

\begin{lemma}\label{lemma:8.1}
  Let $v_k \in \Xkr$ and $v_{kh} \in \Xkhrtwo$ be the semidiscrete
  and fully discrete solutions of \eqref{eq:semidiscrete_heat_measuredata} and 
  \eqref{eq:fully discrete heat}, respectively for $v_0\in L^2(\Om)$. Then there exists a constant $C$ independent
  of $h$,$k$ and $T$ such that
  \begin{equation}
    \|\Delta_h^{-2} \left( P_h v_k - v_{kh} \right)(T)\|_{L^2(\Omega)} \leq
    C h^3 \lk\|v_0\|_{L^2(\Omega)}.
  \end{equation}
\end{lemma}
\begin{proof}
  Let $z_{kh} \in \Xkhrtwo$ be the solution to a dual problem with
  $z_{kh}(T)=\Delta_h^{-2}(P_h v_k - v_{kh})(T)$, i.e.
    \begin{equation*}
      B(\chi_{kh},z_{kh}) = \left( \chi_{kh}(T),\Delta_h^{-2}(P_h v_k - v_{kh})(T)
      \right)
      \quad \text{for all} \quad \chi_{kh} \in \Xkhrtwo.
    \end{equation*}
    Choosing $\chi_{kh} = \Delta_h^{-2}(P_h v_k - v_{kh})$,
    \textcolor{red}{and using the Galerkin orthogonality
    \eqref{eq: orthogonality fully discrete} of $v_k$ and $v_{kh}$, we obtain}
    \begin{equation*}
      \mathcal{Z} := \|\Delta_h^{-2}(P_h v_k - v_{kh})(T)\|^2_{L^2(\Omega)}
      = B(\Delta_h^{-2}(P_h v_k - v_{kh}),z_{kh})
      = B(P_h v_k - v_{kh},\Delta_h^{-2} z_{kh})
      \textcolor{red}{= B(P_h v_k - v_{k},\Delta_h^{-2} z_{kh}).}
    \end{equation*}
    \textcolor{red}{
      Note, that since $z_{kh}$ is piecewise polynomial in time, with values in $V_h^2$,
      for every $t \in I$, it holds $\Delta_h^{-2}z_{kh}(t) \in V_h^2$ and for every
      $t \in I \setminus \{t_0,t_1, ..., t_M\}$ it holds $\partial_t \Delta_h^{-2} z_{kh}(t) \in V_h^2$.
      Using the dual representation of $B$, given in \eqref{eq:B_dual}, and 
      the definition $P_h$, all $L^2(\Om)$ inner products vanish, and it holds
    }
    \begin{equation*}
      \mathcal{Z} = (\nabla(P_h v_k - v_{k}),\nabla(\Delta_h^{-2} z_{kh}))_{I\times \Omega}.
    \end{equation*}
    In this inner product, we can replace $v_k$ by its Ritz projection $R_h v_k$ and obtain after
    applying the definitions of $\Delta_h$, $P_h$ and the duality pairing
    \begin{equation*}
    \begin{aligned}
     \mathcal{Z} =(\nabla(P_h v_k - R_h v_{k}),
     \nabla(\Delta_h^{-2} z_{kh}))_{I\times \Omega}
    & = -(P_h v_k - R_h v_{k},\Delta_h^{-1} z_{kh})_{I\times \Omega}
     = -(v_k - R_h v_{k},\Delta_h^{-1} z_{kh})_{I\times \Omega}\\
    &\le \int_I \|(v_k - R_h v_k)(t)\|_{H^{-1}(\Omega)} \|\Delta_h^{-1}z_{kh}(t)\|_{H^1_0(\Omega)} \ dt.
      \end{aligned}
    \end{equation*}
    The second term in the integral for each fixed $t$, can be estimated as follows,
    \begin{align*}
      \|\Delta_h^{-1}z_{kh}(t)\|_{H^1_0(\Omega)}^2
      &\leq C \left( \nabla \Delta_h^{-1}z_{kh}(t),\nabla \Delta_h^{-1}z_{kh}(t) \right)_{\Om}\\
      & =  - C \left( z_{kh}(t),\Delta_h^{-1}z_{kh}(t) \right)_{\Om}\\
      & \leq  C \|z_{kh}(t)\|_{L^2(\Omega)}\|\Delta_h^{-1}z_{kh}(t)\|_{L^2(\Omega)}\\
      & \leq  C \|z_{kh}(t)\|_{L^2(\Omega)}\|\Delta_h^{-1}z_{kh}(t)\|_{H^1_0(\Omega)},
    \end{align*}
    yielding $\|\Delta_h^{-1}z_{kh}(t)\|_{H^1_0(\Omega)}
    \leq C \|z_{kh}(t)\|_{L^2(\Omega)}$ for almost all $t$.
    Using this estimate together with Lemma \ref{lemm:H-1estimate}, we get
    \begin{equation*}
     \mathcal{Z} \leq C h^3\int_I \|\Delta v_k(t)\|_{L^2(\Omega)}
     \|z_{kh}(t)\|_{L^2(\Omega)} \ dt\leq C h^3 \|\Delta v_k\|_{L^1(I;L^2(\Omega))}
     \|z_{kh}\|_{L^\infty(I;L^2(\Omega))}.
    \end{equation*}
    Using Corollary \ref{cor: maximal parabolic initial in L1}, we finally obtain
        \begin{align*}
     \mathcal{Z} &\leq C h^3 \lk \|v_0\|_{L^2(\Omega)}
            \|z_{kh}(T)\|_{L^2(\Omega)}
      \leq C h^3 \lk \|v_0\|_{L^2(\Omega)}
      \|\Delta_h^{-2}(P_h v_k - v_{kh})(T)\|_{L^2(\Omega)}.
    \end{align*}
Canceling $\|\Delta_h^{-2}(P_h v_k - v_{kh})(T)\|_{L^2(\Omega)}$ gives the result.
\end{proof}

Using this auxiliary result, we can prove the next lemmas estimating $R_h v_k - v_{kh}$.
\begin{lemma}\label{lemma:8.2}
  Let $v_k \in \Xkr$ and $v_{kh} \in \Xkhrtwo$ be the semidiscrete and fully discrete
  solutions of \eqref{eq:semidiscrete_heat_measuredata} and \eqref{eq:fully discrete heat}, respectively for $v_0\in L^2(\Om)$.
  There exists a constant $C$ independent of $k$,$h$, and $T$ such that
  \begin{equation*}
    \| \left( R_h v_k - v_{kh} \right)(T)\|_{\Ltwo} \leq
    C h^3 \left( \frac{1}{T^2} + \frac{1}{T^{\frac{3}{2}}} \right)
    \lk \|v_0\|_\Ltwo.
  \end{equation*}
\end{lemma}
\begin{proof}
  Let $y_{kh} \in \Xkhrtwo$ be the solution to a dual problem with
  $y_{kh}(T) = \left( R_h v_k - v_{kh} \right)(T)$, i.e $y_{kh} \in \Xkhrtwo$
  satisfies
  \begin{equation*}
    B(\varphi_{kh},y_{kh}) = \left( \varphi_{kh}(T), (R_h v_k - v_{kh})(T) \right)
    \quad \text{for all} \quad \varphi_{kh} \in \Xkhrtwo.
  \end{equation*}
  To simplify notation, we define $\psi_{kh} := R_h v_k - v_{kh} \in \Xkhrtwo$.
  We introduce $\tilde{\psi}_{kh} \in \Xkhrtwo$ to be zero on
  $I_1 \cup ... \cup I_{\tilde{m}}$ and $\tilde{\psi}_{kh} = \psi_{kh}$ on
  $I_{\tilde{m}+1} \cup ... I_M$ for $\tilde{m}$ chosen such that
  $ \frac{T}{2} \in I_{\tilde{m}}$. Analogously we define $\tilde{y}_{kh}$.
  Choosing $\tilde{\psi}_{kh}$ as test function in the definition of
  $y_{kh}$ and transfering the
  cutoff from one argument to the other, by (\ref{eq: Bilinear tilde}), we get
  \begin{align*}
    \| \left( R_h v_k - v_{kh} \right)(T)\|_{\Ltwo}^2 &= B(\tilde{\psi}_{kh},y_{kh})\\
      &=B(\psi_{kh},\tilde{y}_{kh}) + 
      \left( \psi_{kh,\tilde{m}}^-,y_{kh,\tilde{m}}^+\right)\\
      &=B(R_h v_k - v_{kh},\tilde{y}_{kh}) + 
        \left( (R_h v_{k,\tilde{m}} - v_{kh,\tilde{m}})^-,
        y_{kh,\tilde{m}}^+\right)\\
      &=B(R_h v_k - v_{k},\tilde{y}_{kh}) + 
        \left( (R_h v_{k,\tilde{m}} - v_{kh,\tilde{m}})^-,
        y_{kh,\tilde{m}}^+\right)\\
      &= J_1 + J_2.
  \end{align*}
  Here we also have used the Galerkin orthogonality \eqref{eq: orthogonality fully discrete} with respect
  to the bilinear form $B$. By the definition of the Ritz projection the terms 
  $\left(\nabla (R_h v_k - v_{k}), \nabla \tilde{y}_{kh}\right)_{I_m \times \Omega}$
  vanish from the form $B$, such that the remaining terms in $J_1$ are
  \begin{align*}
    J_1 &= -\smashoperator{\sum_{m=\tilde{m}+1}^M} (R_h v_k - v_k,
    \partial_t y_{kh})_{I_m\times\Omega}
    - \smashoperator{\sum_{m=\tilde{m}+1}^M} (R_h v_{k,m}^- - v_{k,m}^-,[y_{kh}]_m)
    - \left( R_h v_{k,\tilde{m}}^- - v_{k,\tilde{m}}^-,y_{kh,\tilde{m}}^+ \right)\\
    &\leq \|R_h v_k - v_k\|_{L^\infty((t_{\tilde{m}-1},T);\Ltwo)}
    \left( \|\partial_t y_{kh}\|_{L^1(I;\Ltwo)} +
      \smashoperator{\sum_{m=1}^M}\|[y_{kh}]_m\|_{\Ltwo} + \|y_{kh,\tilde{m}}^+\|_{\Ltwo}
    \right),
  \end{align*}
  where we used the dual form of $B(\cdot,\cdot)$ and H\"olders inequality in space and time to 
  estimate the terms.
  Applying Corollary \ref{cor: maximal parabolic initial in L1},  gives
  \begin{equation*}
    \|\partial_t y_{kh}\|_{L^1(I;\Ltwo)} + 
    \sum_{m=1}^M\|[y_{kh}]_m\|_{\Ltwo} +
    \|y_{kh,\tilde{m}}^+\|_{\Ltwo}
    \leq C \lk \|(R_h v_k - v_{kh})(T)\|_{\Ltwo}.
  \end{equation*}
  Note that $y_{kh}$ is a solution to a dual problem and we use $y_{kh}(T)$ as bound. Using the $L^2$ error
  estimate for the Ritz projection, together with the estimate \eqref{eq:est:H3} of 
  Lemma \ref{thm:h3regularity}, we obtain for any $t \in (t_{\tilde{m}-1},T]$
  \begin{equation*}
   \ \|(R_h v_k - v_k)(t)\|_{\Ltwo} \leq C h^3
      \| v_k(t)\|_{H^3(\Omega)}\le Ch^3 
      \|\Delta v_k(t)\|_{\Ltwo}^{ \frac{1}{2}}
    \|\Delta^2 v_k(t)\|_{\Ltwo}^{ \frac{1}{2}}.
  \end{equation*}
  Using the smoothing results of Lemma \ref{lemma: higher smoothing}, we obtain
  \begin{equation}\label{eq:ritzerror_over_time}
  \begin{aligned}
   \sup_{t\in(t_{\tilde{m}-1},T]} \|(R_h v_k - v_k)(t)\|_{\Ltwo)} &\leq
   C h^3 \sup_{t\in(t_{\tilde{m}-1},T]} \|\Delta v_k(t)\|_{\Ltwo}^{ \frac{1}{2}}
    \|\Delta^2 v_k(t)\|_{\Ltwo}^{ \frac{1}{2}}\\
       &\leq C \dfrac{h^3}{t_{\tilde{m}}^{\frac{3}{2}}} \|v_0\|_{\Ltwo}
      \leq C \dfrac{h^3}{T^{\frac{3}{2}}} \|v_0\|_{\Ltwo}.
  \end{aligned}
  \end{equation}
  In the last step, we used the estimate $ \frac{1}{t_{\tilde{m}}} \leq \frac{2}{T}$
  which holds true, since $t_{\tilde{m}}$ was chosen such that
  $ \frac{T}{2} \in I_{\tilde{m}}$, and thus, $ \frac{T}{2}\leq t_{\tilde{m}}$.
  Combining these results gives the proposed estimate for $J_1$:
  \begin{equation*}
    J_1 \leq C \dfrac{h^3}{T^{ \frac{3}{2}}} 
    \lk\|v_0\|_{\Ltwo}
    \|(R_h v_k - v_{kh})(T)\|_{\Ltwo}.
  \end{equation*}
  To estimate $J_2$ we insert an artificial zero by adding and subtracting
  $v_{k,\tilde{m}}^-$,
  \begin{align*}
    J_2 &= ((R_h v_{k,\tilde{m}} - v_{kh,\tilde{m}})^-,y_{kh,\tilde{m}}^+)=((R_h v_{k,\tilde{m}} - v_{k,\tilde{m}})^-,y_{kh,\tilde{m}}^+) +
    ((v_{k,\tilde{m}} - v_{kh,\tilde{m}})^-,y_{kh,\tilde{m}}^+)
    := J_{21} + J_{22}.
  \end{align*}
  The term $J_{21}$ can be estimated with \eqref{eq:ritzerror_over_time},
  the discrete smoothing result of Lemma \ref{lemma: higher smoothing} applied to
  $\|y_{kh,\tilde{m}}^+\|_\Ltwo$ and the special choice of $\tilde{m}$:
  \begin{align*}
    J_{21} &\leq \sup_{t\in(t_{\tilde{m}-1},T)} \|(R_h v_k - v_k)(t)\|_{\Ltwo)}
    \|y_{kh,\tilde{m}}^+\|_{\Ltwo}\\
    &\leq C \dfrac{h^3}{t_{\tilde{m}}^{ \frac{3}{2}}} \|v_0\|_{\Ltwo}
    \|y_{kh}(T)\|_{\Ltwo}\leq C \dfrac{h^3}{T^{ \frac{3}{2}}} \|v_0\|_{\Ltwo}
    \|(R_h v_k - v_{kh})(T)\|_{\Ltwo}.
  \end{align*}
  To estimate $J_{22}$ we use Lemma \ref{lemma:8.1} by using
  $t_{\tilde{m}}$ as artificial endtime. Here it is of importance, that the
  derived constant does not depend on the endtime, since we need to replace 
  $t_{\tilde{m}}$ by $T$ later. This can only be done, when the explicit dependence of
  the result of Lemma \ref{lemma:8.1} on the endtime is known. We thus get after inserting
  the $L^2$ projection operator:
  \begin{align*}
    J_{22} &= ((P_hv_{k,\tilde{m}} - v_{kh,\tilde{m}})^-,y_{kh,\tilde{m}}^+)\\
    &= (\Delta_h^{-2}(P_h v_{k,\tilde{m}} - v_{kh,\tilde{m}})^-,
        \Delta_h^2 y_{kh,\tilde{m}}^+)\\
    &\leq \|\Delta_h^{-2}(P_h v_{k,\tilde{m}} - v_{kh,\tilde{m}})^-\|_{\Ltwo}
        \|\Delta_h^2 y_{kh,\tilde{m}}^+\|_{\Ltwo}\\
    &\leq C \ln \left( \frac{t_{\tilde{m}}}{k}\right)\|v_0\|_{\Ltwo}
    \frac{h^3}{{(T-t_{\tilde{m}})}^2}\|(R_h v_k - v_{kh})(T)\|_{\Ltwo}.
  \end{align*}
  In the last step, we have used Lemma \ref{lemma:8.1} for 
  $\|\Delta_h^{-2}(P_h v_{k,\tilde{m}} - v_{kh,\tilde{m}})^-\|_{\Ltwo}$ and the 
  discrete smoothing result of Lemma \ref{lemma: higher smoothing} for 
  $\|\Delta_h^2 y_{kh,\tilde{m}}^+\|_{\Ltwo}$. Since $y_{kh}$ is a solution to a 
  backwards problem, we use $\frac{1}{T- t_{\tilde{m}}}$, instead of 
  $\frac{1}{t_{\tilde{m}}}$. We now replace all occurrences of $t_{\tilde{m}}$ by $T$.
  As before, we use $\frac{T}{2} \in I_{\tilde{m}}$ yielding
  $ t_{\tilde{m}}\leq \frac{T}{2} + k$. For fine enough time discretizations
  (i.e. $ \frac{T}{4}>k$) we have
  \begin{equation*}
    T-t_{\tilde{m}} \ge T - \frac{T}{2} - k \ge T - \frac{T}{2} - \frac{T}{4} 
    = \frac{T}{4},
  \end{equation*}
  thus giving $\frac{1}{T-t_{\tilde{m}}} \le C \frac{1}{T}$.
  To estimate the logarithmic term, we use the following consideration:
  Let $x \in \mathbb{R}$ such that $x > 2$. Then it holds $x+1 \le x^2$. With the 
  monotonicity of the logarithm, we obtain $\ln(x+1) \le \ln(x^2) = 2 \ln(x)$.
  Applying this to the logarithmic term, wile using $ t_{\tilde{m}}\leq \frac{T}{2} + k$
  and $\frac{T}{2k}>2$, yields
  \begin{equation*}
    \ln \left( \frac{t_{\tilde{m}}}{k}\right) \leq 
    \ln \left( \frac{ \frac{T}{2}+k}{k}\right) =
    \ln \left( \frac{T}{2k}+1\right) \leq 2\ln \left( \frac{T}{2k} \right)
    \leq 2\ln \left( \frac{T}{k} \right).
  \end{equation*}
  This gives a bound for $J_{22}$, depending on the final time $T$,
  \begin{equation*}
    J_{22} \leq C \frac{h^3}{T^2}\lk
    \|v_0\|_{\Ltwo} \|(R_h v_k - v_{kh})(T)\|_{\Ltwo}.
  \end{equation*}
  Dividing all considered terms by $\|(R_h v_k - v_{kh})(T)\|_{\Ltwo}$ gives the
  proposed estimate.
\end{proof}
We now show a similar result for the discrete Laplacian.
\begin{lemma}\label{lemma:8.3}
  Let $v_k \in \Xkr$ and $v_{kh} \in \Xkhrtwo$ be the semidiscrete and fully
  discrete solutions of \eqref{eq:semidiscrete_heat_measuredata} and \eqref{eq:fully discrete heat},
  respectively. Then there exists a constant independent of $k$, $h$, and $T$ such that
  \begin{equation*}
    \|\Delta_h(R_h v_k - v_{kh})(T)\|_{\Ltwo} \leq C h^3
    \left( \frac{1}{T^3} + \frac{1}{T^{ \frac{5}{2}}} \right)
    \lk  \|v_0\|_{\Ltwo}.
  \end{equation*}
\end{lemma}
\begin{proof}
  Let $y_{kh} \in \Xkhrtwo$ be the solution to a dual problem with
  $y_{kh}(T) = \Delta_h \left( R_h v_k - v_{kh} \right)(T)$, i.e $y_{kh} \in \Xkhrtwo$
  satisfies
  \begin{equation}
    B(\varphi_{kh},y_{kh}) = \left( \varphi_{kh}(T), \Delta_h(R_h v_k - v_{kh})(T) \right)
    \quad \text{for all} \quad \varphi_{kh} \in \Xkhrtwo.
  \end{equation}
  As in the previous lemma, in order to simplify notation,
  we define ${\psi_{kh} := R_h v_k - v_{kh} \in \Xkhrtwo}$.
  We introduce $\tilde{\psi}_{kh} \in \Xkhrtwo$ to be zero on
  $I_1 \cup ... \cup I_{\tilde{m}}$ and $\tilde{\psi}_{kh} = \psi_{kh}$ on
  $I_{\tilde{m}+1} \cup ... I_M$ for $\tilde{m}$ chosen such that
  $ \frac{T}{2} \in I_{\tilde{m}}$. We define $\tilde{y}_{kh}$ analogously.
  Choosing $\Delta_h \tilde{\psi}_{kh}$ as test function in the definition of $y_{kh}$,
  and transfering the cutoff from the first argument of $B$ to the second, by applying
  \eqref{eq: Bilinear tilde}, we get
  \begin{align*}
    \|\Delta_h \left( R_h v_k - v_{kh} \right)(T)\|_{\Ltwo}^2
    &= B(\Delta_h\tilde{\psi}_{kh},y_{kh})\\
    &= B(\tilde{\psi}_{kh},\Delta_hy_{kh})\\
    &= B(\psi_{kh},\Delta_h\tilde{y}_{kh}) + 
      \left( \psi_{kh,\tilde{m}}^-,\Delta_hy_{kh,\tilde{m}}^+\right)\\
    &= B(R_h v_k - v_{kh},\Delta_h\tilde{y}_{kh}) + 
        \left( (R_h v_{k,\tilde{m}} - v_{kh,\tilde{m}})^-,
        \Delta_hy_{kh,\tilde{m}}^+\right)\\
    &=B(R_h v_k - v_{k},\Delta_h\tilde{y}_{kh}) + 
        \left( (R_h v_{k,\tilde{m}} - v_{kh,\tilde{m}})^-,
        \Delta_hy_{kh,\tilde{m}}^+\right)\\
    &= J_1 + J_2.
  \end{align*}
  Here we have used the Galerkin orthogonality \eqref{eq: orthogonality fully discrete}
  with respect to the bilinear form $B$. By the definition of the Ritz projection the terms 
  $\left(\nabla (R_h v_k - v_{k}),
  \nabla \Delta_h\tilde{y}_{kh}\right)_{I_m \times \Omega}$
  vanish from the form $B$, such that the remaining terms of $J_1$ are
  \begin{equation*}
    J_1 = \sum_{m=\tilde{m}+1}^M
    (\partial_t (R_h v_k - v_k),\Delta_h y_{kh})_{I_m\times\Omega}
    + \sum_{m=\tilde{m}+1}^M ([R_h v_k -v_k]_m,\Delta_h y_{kh,m}^+).
  \end{equation*}
  Applying H\"older's inequality in space and time gives
  \begin{align*}
    J_1 \le
    \|\partial_t (R_h v_k - v_k)\|_{L^\infty((t_{\tilde{m}},T);\Ltwo)}
    \|\Delta_h y_{kh}\|_{L^1((t_{\tilde{m}},T);\Ltwo)}
    + \sum_{m=\tilde{m}+1}^M
    \|[R_h v_k -v_k]_m\|_{\Ltwo}
    \|\Delta_h y_{kh,m}^+\|_{\Ltwo}.
  \end{align*}
  Introducing an artificial factor $1$ as $k_m \cdot k_m^{-1}$ in the sum allows us to extract
  the term $ \max_{\tilde{m}\le m \le M} \left\lbrace k_m^{-1} \|[R_h v_k -v_k]_m\|_{\Ltwo} \right\rbrace$
  out of the sum. This gives
  \begin{align*}
    J_1 \le& 
    \max_{\tilde{m}\le m\le M}\|\partial_t (R_h v_k - v_k)\|_{L^\infty(I_m;\Ltwo)}
    \|\Delta_h y_{kh}\|_{L^1((t_{\tilde{m}},T);\Ltwo)}\\
    &+ \max_{\tilde{m}\le m \le M} 
    \left\lbrace k_m^{-1} \|[R_h v_k -v_k]_m\|_{\Ltwo} \right\rbrace
    \left(\sum_{m=\tilde{m}+1}^M k_m \|\Delta_h y_{kh,m}^+\|_{\Ltwo}\right).
  \end{align*}
Using  Corollary \ref{cor: maximal parabolic initial in L1} and  the $L^2$ error estimate for the Ritz projection for the other terms, gives the following estimate,
  \begin{align*}
    J_1 \le C h^3 \left( 
      \max_{\tilde{m}\le m\le M}\|\partial_t v_k\|_{L^\infty(I_m;H^3(\Omega))}
    + \max_{\tilde{m}\le m \le M} 
    \left\lbrace k_m^{-1} \|[v_k]_m\|_{H^3(\Omega))} \right\rbrace\right)
    \lk\|\Delta_h(R_h v_k - v_{kh})(T)\|_{\Ltwo}.
  \end{align*}
  Similar to the previous lemma, by the estimate \eqref{eq:est:H3} of Lemma \ref{thm:h3regularity}
  and Remark \ref{rem:bc_laplacian_semidiscrete}, yielding $\partial_t \Delta v_k(t) \in H^1_0(\Omega)$
  for $t\in (t_{\tilde{m}-1},T]$, we obtain
  \begin{equation*}
    \|\partial_t v_k(t)\|_{H^3(\Omega)} 
    \le C\|\partial_t \Delta v_k(t)\|^{ \frac{1}{2}}_{\Ltwo} 
    \|\partial_t \Delta^2 v_k(t)\|^{ \frac{1}{2}}_{\Ltwo}.
  \end{equation*}
  Taking the supremum over $(t_{\tilde{m}-1},T]$ and using Lemma \ref{lemma: higher smoothing} then yields
  \begin{align*}
      \max_{\tilde{m}\le m\le M}\|\partial_t v_k(t)\|_{L^\infty(I_m;H^3(\Omega))} 
      &\leq
      C
       \max_{\tilde{m}\le m\le M}\left\|\partial_t\Delta v_k(t)
      \right\|^{ \frac{1}{2}}_{L^\infty(I_m;\Ltwo)}
      \left\|\partial_t\Delta^2 v_k(t)
    \right\|^{ \frac{1}{2}}_{L^\infty(I_m;\Ltwo)}\\
      &\leq
      C \dfrac{1}{t_{\tilde{m}}^\frac{5}{2}} \|v_0\|_{\Ltwo} 
      \le C \dfrac{1}{T^\frac{5}{2}} \|v_0\|_{\Ltwo}.
  \end{align*}
  Applying the same arguments, using $\Delta v_k(t) \in H^1_0(\Omega)$ for 
  $t \in (t_{\tilde{m}-1},T)$, and Lemma \ref{lemma: higher smoothing} gives
  \begin{align*}
    \max_{\tilde{m}\le m \le M} 
    \left\lbrace k_m^{-1} \|[v_k]_m\|_{H^3(\Omega)} \right\rbrace
    \le
    \max_{\tilde{m}\le m \le M} 
    \left\lbrace k_m^{-1} \left(\|\Delta[v_k]_m\|_{\Ltwo}\|\Delta^2[v_k]_m\|_{\Ltwo}\right)^{ \frac{1}{2}}
    \right\rbrace
    \leq
    C \dfrac{1}{t_{\tilde{m}}^\frac{5}{2}} \|v_0\|_{\Ltwo} 
    \le C \dfrac{1}{T^\frac{5}{2}} \|v_0\|_{\Ltwo}.
  \end{align*}
  Summarizing all above results yields the final bound for $J_1$:
  \begin{equation*}
    J_1 \le C h^3 \dfrac{1}{T^\frac{5}{2}} \lk \|v_0\|_{\Ltwo} 
    \|\Delta_h(R_h v_k - v_{kh})(T)\|_{\Ltwo}.
  \end{equation*}
  To estimate $J_2$ we insert an artificial zero like before by adding and subtracting
  $v_{k,\tilde{m}}^-$:
  \begin{align*}
    J_2 &= ((R_h v_{k,\tilde{m}} - v_{kh,\tilde{m}})^-,\Delta_hy_{kh,\tilde{m}}^+)=((R_h v_{k,\tilde{m}} - v_{k,\tilde{m}})^-,\Delta_hy_{kh,\tilde{m}}^+) +
    ((v_{k,\tilde{m}} - v_{kh,\tilde{m}})^-,\Delta_hy_{kh,\tilde{m}}^+)
    := J_{21} + J_{22}.
  \end{align*}
  The term $J_{21}$ can be estimated similarly to the previous lemma,
  applying \eqref{eq:ritzerror_over_time}, the discrete smoothing result of Lemma
  \ref{lemma: higher smoothing} for $\|\Delta_h y_{kh,\tilde{m}}^+\|_\Ltwo$ and using the special
  choice of $\tilde{m}$:
  \begin{align*}
    J_{21} &\le \sup_{t\in(t_{\tilde{m}-1},T)}\|(R_h v_k - v_k)(t)\|_{\Ltwo}
    \|\Delta_hy_{kh,\tilde{m}}^+\|_{\Ltwo}\\
    &\le C \dfrac{h^3}{t_{\tilde{m}}^{\frac{3}{2}}} \|v_0\|_{\Ltwo}
    \frac{1}{T - t_{\tilde{m}}} \|y_{kh}(T)\|_{\Ltwo}\\
    &\le C \dfrac{h^3}{T^{ \frac{5}{2}}} \|v_0\|_{\Ltwo}
    \|\Delta_h(R_h v_k - v_{kh})(T)\|_{\Ltwo}.
  \end{align*}
  We estimate $J_{22}$ by replacing $v_{k,\tilde{m}}$ with its $L^2$-projection:
  \begin{align*}
    J_{22} &= ((P_hv_{k,\tilde{m}} - v_{kh,\tilde{m}})^-,\Delta_hy_{kh,\tilde{m}}^+)\\
    &= (\Delta_h^{-2}(P_h v_{k,\tilde{m}} - v_{kh,\tilde{m}})^-,
        \Delta_h^3 y_{kh,\tilde{m}}^+)\\
    &\leq \|\Delta_h^{-2}(P_h v_{k,\tilde{m}} - v_{kh,\tilde{m}})^-\|_{\Ltwo}
        \|\Delta_h^3 y_{kh,\tilde{m}}^+\|_{\Ltwo}\\
    &\leq C \dfrac{h^3}{(T-t_{\tilde{m}})^3}
    \ln \left( \frac{t_{\tilde{m}}}{k}\right)\|v_0\|_{\Ltwo}
    \|\Delta_h(R_h v_k - v_{kh})(T)\|_{\Ltwo}.
  \end{align*}
  In the last step, we have used Lemma \ref{lemma:8.1} for
  $\|\Delta_h^{-2}(P_h v_{k,\tilde{m}} - v_{kh,\tilde{m}})^-\|_{\Ltwo}$
  and the discrete smoothing result of Lemma \ref{lemma: higher smoothing} for 
  $\|\Delta_h^3 y_{kh,\tilde{m}}^+\|_{\Ltwo}$. Since $y_{kh}$ is the solution to a
  problem backward in time, we use $ \frac{1}{T- t_{\tilde{m}}}$ instead of
  $ \frac{1}{t_{\tilde{m}}}$ in the application of this result.
  Analogously to the previous lemma, we can replace the terms involving $t_{\tilde{m}}$
  by ones dependent only of $T$ because of the special choice of $t_{\tilde{m}}$,
  thus giving the final bound for $J_{22}$,
  \begin{equation*}
    J_{22} \leq C \frac{h^3}{T^3}\lk
    \|v_0\|_{\Ltwo} \|\Delta_h(R_h v_k - v_{kh})(T)\|_{\Ltwo}.
  \end{equation*}
  Dividing all considered terms by $\|\Delta_h(R_h v_k - v_{kh})(T)\|_{\Ltwo}$ gives the
  proposed estimate.
\end{proof}
Combining Lemma \ref{lemma:8.2} and Lemma \ref{lemma:8.3}  with the discrete
Gagliardo-Nirenberg inequality 
\eqref{eq: discrete Gagliardo-Nirenber} gives the following result:
\begin{corr}\label{cor: Rhvk-vkh Linfty to L2}
  Let $v_k \in \Xkr$ and $v_{kh} \in \Xkhrtwo$ be the semidiscrete and fully discrete solutions of 
 of \eqref{eq:eq_aux} with  $v_0\in L^2(\Om)$. Then there exists a constant independent of $k$ and $h$ such that
  \begin{equation*}
    \|(R_h v_k - v_{kh})(T)\|_{L^\infty(\Omega)} \le C h^3 \lk
    \left( \frac{1}{T^3} + \frac{1}{T^\frac{5}{2}} \right)^\frac{N}{4}
    \left( \frac{1}{T^2} + \frac{1}{T^\frac{3}{2}} \right)^{\left(1-\frac{N}{4}\right)}
    \|v_0\|_{\Ltwo}.
  \end{equation*}
\end{corr}
This result now allows us to estimate the final term $T_3$ of \eqref{eq:err_split_quadFE} and thus 
proves Theorem \ref{thm:smoothing_est_space_p2}.

\subsection{Estimates for $(v-v_{kh})(T)$ with $v_0\in \M(\Om)$}\label{subsec: measure}

Now that we have established Theorem \ref{thm:smoothing_est_space_p2} for $v_0\in L^2(\Omega)$,
following exactly the proof of Theorem \ref{thm:fully discrete from Linfty_to_measure},
and using the Assumption \ref{assump:regularityH3},
we can establish 

\begin{theorem}\label{thm:smoothing_est_space_p2_measure}
Let $\Omega_0$ be a subdomain with $\overline \Omega_0 \subset \Omega$,
$v_0\in \M(\Om)$ with $\supp v_0 \subset \Omega_0$. 
Let $v$ and $v_{kh} \in \Xkhrtwo$ satisfy \eqref{eq:eq_aux} and 
\eqref{eq:fully discrete heat}, respectively. In addition, let $\Omega$ be such that 
Assumption \ref{assump:regularityH3} holds.
\begin{equation}\label{eq:h3_regularity}
    \|(v-v_{kh})(T)\|_{L^\infty(\Om_0)} \le C(T,\Om_0)\left( \ell_k h^3 + k^{2r+1}\right)\|v_0\|_{\M(\Om)},
\end{equation}
where $\ell_{k} = \lk$, $d=\dist(\Om_0,\partial \Om)$
and $C(T,d)$ is a constant depending on $T$ and $d$.
\end{theorem}

\bibliographystyle{siam}
\bibliography{sources}

\begin{thebibliography}{10}

\bibitem{AdamsFournier:1977}
{\sc R.~A. Adams and J.~Fournier}, {\em Cone conditions and properties of
  {S}obolev spaces}, J. Math. Anal. Appl., 61 (1977), pp.~713--734.

\bibitem{Adams_Fournier_2005}
{\sc R.~A. Adams and J.~J.~F. Fournier}, {\em Sobolev spaces}, vol.~140 of Pure
  and Applied Mathematics (Amsterdam), Elsevier/Academic Press, Amsterdam,
  second~ed., 2003.

\bibitem{ashyralyev_wellposedness_1994}
{\sc A.~Ashyralyev and P.~E. Sobolevskii}, {\em Well-Posedness of Parabolic
  Difference Equations}, Birkhäuser Basel, 1994.

\bibitem{BrennerScott}
{\sc S.~C. Brenner and L.~R. Scott}, {\em The mathematical theory of finite
  element methods}, vol.~15 of Texts in Applied Mathematics, Springer, New
  York, third~ed., 2008.

\bibitem{CasasE_VexlerB_ZuazuaE_2015a}
{\sc E.~Casas, B.~Vexler, and E.~Zuazua}, {\em Sparse initial data
  identification for parabolic {PDE} and its finite element approximations},
  Math. Control Relat. Fields, 5 (2015), pp.~377--399.

\bibitem{drelichman_weighted_2020}
{\sc I.~Drelichman, R.~G. Durán, and I.~Ojea}, {\em A weighted setting for the
  numerical approximation of the poisson problem with singular sources}, {SIAM}
  J. Numer. Anal., 58 (2020), pp.~590--606.

\bibitem{ErikssonK_JohnsonC_LarssonS_1998a}
{\sc K.~Eriksson, C.~Johnson, and S.~Larsson}, {\em Adaptive finite element
  methods for parabolic problems. {VI}. {A}nalytic semigroups}, SIAM J. Numer.
  Anal., 35 (1998), pp.~1315--1325.

\bibitem{evans_partial_2010}
{\sc L.~C. Evans}, {\em Partial differential equations}, no.~19 in Graduate
  studies in mathematics, American Mathematical Society, second~ed., 2010.

\bibitem{HansboA_2002a}
{\sc A.~Hansbo}, {\em Strong stability and non-smooth data error estimates for
  discretizations of linear parabolic problems}, BIT, 42 (2002), pp.~351--379.

\bibitem{HellOstermannSandbichler:2015}
{\sc T.~Hell, A.~Ostermann, and M.~Sandbichler}, {\em Modification of
  dimension-splitting methods---overcoming the order reduction due to corner
  singularities}, IMA J. Numer. Anal., 35 (2015), pp.~1078--1091.

\bibitem{lesaint_finite_1974}
{\sc P.~Lesaint and P.~Raviart}, {\em On a finite element method for solving
  the neutron transport equation}, in Mathematical Aspects of Finite Elements
  in Partial Differential Equations, Academic Press, pp.~89--123.

\bibitem{LeykekhmanD_VexlerB_2016a}
{\sc D.~Leykekhman and B.~Vexler}, {\em Pointwise best approximation results
  for {G}alerkin finite element solutions of parabolic problems}, SIAM J.
  Numer. Anal., 54 (2016), pp.~1365--1384.

\bibitem{LeykekhmanD_VexlerB_2017a}
\leavevmode\vrule height 2pt depth -1.6pt width 23pt, {\em Discrete maximal
  parabolic regularity for {G}alerkin finite element methods}, Numer. Math.,
  135 (2017), pp.~923--952.

\bibitem{LeyVexWagl_2023}
{\sc D.~Leykekhman, B.~Vexler, and J.~Wagner}, {\em Numerical analysis of
  sparse initial data identification for parabolic problems with pointwise
  final time observations}, in preparation,  (2023).

\bibitem{LeyVexWal19}
{\sc D.~Leykekhman, B.~Vexler, and D.~Walter}, {\em Numerical analysis of
  sparse initial data identification for parabolic problems}, ESAIM Math.
  Model. Numer. Anal., 54 (2020), pp.~1139--1180.

\bibitem{DMeidner_RRannacher_BVexler_2011a}
{\sc D.~Meidner, R.~Rannacher, and B.~Vexler}, {\em A priori error estimates
  for finite element discretizations of parabolic optimization problems with
  pointwise state constraints in time}, SIAM J. Control Optim., 49 (2011),
  pp.~1961--1997.

\bibitem{saff_zeros_1975}
{\sc E.~B. Saff and R.~S. Varga}, {\em On the zeros and poles of padé
  approximants to e{\textasciicircum}z}, Numerische Mathematik, 25 (1975),
  pp.~1--14.

\bibitem{AHSchatz_VThomee_LBWahlbin_1980a}
{\sc A.~H. Schatz, V.~C. Thom{\'e}e, and L.~B. Wahlbin}, {\em Maximum norm
  stability and error estimates in parabolic finite element equations}, Comm.
  Pure Appl. Math., 33 (1980), pp.~265--304.

\bibitem{SchatzWahlbin_1977}
{\sc A.~H. Schatz and L.~B. Wahlbin}, {\em Interior maximum norm estimates for
  finite element methods}, Math. Comp., 31 (1977), pp.~414--442.

\bibitem{ThomeeV_2006}
{\sc V.~Thom\'{e}e}, {\em Galerkin finite element methods for parabolic
  problems}, vol.~25 of Springer Series in Computational Mathematics,
  Springer-Verlag, Berlin, second~ed., 2006.

\bibitem{wanner_order_1978}
{\sc G.~Wanner, E.~Hairer, and S.~P. Nørsett}, {\em Order stars and stability
  theorems}, {BIT}, 18 (1978), pp.~475--489.

\end{thebibliography}


%
\end{document}